\documentclass[a4paper,11pt]{article}
\usepackage{amsmath}
\usepackage{amsfonts}
\usepackage{supertabular}
\usepackage[latin9]{inputenc}
\usepackage[english]{varioref}
\usepackage{dcolumn}
\usepackage[height=22cm , width = 16cm , top = 4cm , left = 3cm, a4paper]{geometry}
\usepackage[a4paper]{geometry}
\usepackage[final]{graphicx}
\usepackage{epsfig}
\usepackage{pstricks}
\usepackage{psfrag}
\usepackage{rotating}
\usepackage{supertabular}
\usepackage{booktabs}
\usepackage{delarray}
\usepackage{rotating}
\usepackage{subfigure}
\usepackage{nextpage}
\usepackage{layout}
\usepackage{amsthm}
\usepackage{dsfont}
\usepackage{hyperref}
\usepackage[hypcap]{caption}
\usepackage{color}
\numberwithin{equation}{section}

\theoremstyle{plain}
\newtheorem{thm}{Theorem}[section]

\newtheorem{lem}[thm]{Lemma}

\newtheorem{defn}[thm]{Definition}
\newtheorem{rmkk}[thm]{Remark}
\newtheorem{hypp}[thm]{Hypotheses}
\newcommand{\enter}{\bigskip}

\date{November 19, 2017}
\begin{document}
 \author{{Prasanta Kumar Barik and
 Ankik Kumar Giri}\vspace{.2cm}\footnote{Corresponding author. Tel +91-1332-284818 (O);  Fax: +91-1332-273560  \newline{\it{${}$ \hspace{.3cm} Email address: }}ankikgiri.fma@iitr.ac.in/ankik.math@gmail.com}\\
\footnotesize \small{ \textit{Department of Mathematics, Indian Institute of Technology Roorkee,  Roorkee-247667, Uttarakhand,}}\\ \small{\textit{India}}
  }

\title{Well-posedness to the continuous coagulation processes with collision-induced multiple fragmentation }

\maketitle

%%%%%%%%%%%%%%%%%%%%%%%%%%%% %%%%%%%%%%%%%%%%%
\hrule \vskip 11pt

\begin{quote}
{\small {\em\bf Abstract.} An existence result on weak solutions to the continuous coagulation equation with collision-induced multiple fragmentation is established for certain classes of unbounded coagulation, collision and breakup kernels.  In this model, a pair of particles can coagulate into a larger one if their confrontation is a complete inelastic collision; otherwise, one of them will split into many smaller particles due to a destructive collision. In the present work, both coagulation and fragmentation processes are considered to be intrinsically nonlinear. The breakup kernel may have a possibility to attain a singularity at the origin. The proof is based on the classical weak $L^1$ compactness method applied to suitably chosen approximating equations. In addition, we study the uniqueness of weak solutions under additional growth conditions on collision and breakup kernels which mainly relies on the integrability of higher moments. Finally, it is obtained that the unique weak solution is mass-conserving.\enter
}
\end{quote}

\vskip 11pt \hrule

%%%%%%%%%%%%%%%%%%%%%%%%%%%%%%%%%%%%%%%%%%%%%%%%%%%%%%%%%%%%%%%%%%%
%%%%%%%%%%%%%%%%%%%%%%%%%%%%%%%%%%%%%%%%%%%%%%%%%%%%%%%%%%%%%%%%%%%
\section{Introduction}
Coagulation is a kinetic process in which two particles combine to form a bigger particle whereas in fragmentation process a bigger particle splits into small fragments. In general, coagulation event is always a nonlinear process. However, the fragmentation process may be classified into two major categories on the basis of fragmentation behaviour of particles, one of them is \emph{linear fragmentation} and another one is \emph{nonlinear fragmentation}. The linear fragmentation may occur due to external forces or spontaneously (that depends on the nature of particles). However, if the fragmentation behaviour does not depend only on its nature and external agents but also depends on the state and properties of the entire system, in such a situation nonlinear fragmentation occurs. The simplest case of nonlinear fragmentation takes place due to the collision between two particles. Therefore, it is also known as \emph{collision-induced fragmentation}. Linear fragmentation equations are widely studied by many mathematicians using various techniques, see \cite{Dubovskii:1996, Giri:2013, Giri:2010A, Giri:2012, Giri:2011, McLaughlin:1997, Stewart:1989, Stewart:1990}. However, nonlinear fragmentation equations or collision-induced multiple fragmentation equations did not get proper attention in mathematical community. The nonlinear fragmentation equation is quite difficult to handle mathematically. Therefore, in this paper, a new mathematical model on the continuous coagulation equation with collision-induced multiple fragmentation is studied. A discrete version of the coagulation equation with collision-induced binary fragmentation has been already studied by Jianhang et al. \cite{Jianhong:2008}. Such models arise in polymer science, astrophysics and raindrop breakup etc.  In this article, we assume that two particles can coalescence to form a lager particles, when they meet. Meanwhile, if a pair of particles come across and destructively collide with each other, fragmentation of particles occurs.\\

  Hence, the continuous coagulation equation with collision-induced multiple fragmentation for the change in concentration of the particle $g=g(z,t)$ of volume $z \in \mathbb{R}_{+}:=(0, \infty)$ and at time $t \in[0, \infty)$ is given by
\begin{align}\label{cfe}
\frac{\partial g(z,t)}{\partial t}  =& \frac{1}{2}\int_0^z K(z-z_1, z_1)g(z-z_1, t)g(z_1, t)dz_1- \int_{0}^{\infty} K(z,z_1)g(z,t)g(z_1,t)dz_1\nonumber\\
 &+ \int_{z}^{\infty} \int_{0}^{\infty}B(z|z_1;z_2)C(z_1,z_2)g(z_1,t)g(z_2,t)dz_2dz_1\nonumber\\
 &-\int_{0}^{\infty} C(z,z_1)g(z,t)g(z_1,t)dz_1.
\end{align}
In this article, we show the existence and uniqueness of weak solutions to the above nonlinear integro-partial differential equation (\ref{cfe})
with the following initial condition:
\begin{align}\label{in1}
g(z,0) = g_{0}(z)\geq 0~ \mbox{a.e.}
\end{align}
Here, $K(z,z_1)$ denotes the coagulation kernel, which describes the rate at which  particles of volumes $z$ and $z_1$ meet to form bigger particles of volume $z+z_1$, which is symmetric i.e. $K(z, z_1)=K(z_1 ,z),  \forall  z, z_1 \in \mathbb{R}_{+}$ in nature and the collision kernel $C(z, z_1)$ represents the rate of interaction between two particles $z$ and $z_1$, which is also symmetric. The collision kernel is called homogeneous with an index of homogeneity $\lambda$, if $C(az,az_1)=a^{\lambda}C(z,z_1)$. A detailed study on collision kernel is described in \cite{Matthieu:2007, Kostoglou:2000}.\\

The breakup or breakage kernel $B(z|z_1;z_2)$ is a conditional probability function for the formation of particles of volume $z$ resulting from the breakup of particles of volume $z_1$ due to their collision with  particles of volume $z_2$. The breakup kernel is similar to  the breakage function considered in the linear fragmentation equation. It also attains the similar property as the breakage function.\\

The first and the second integrals on the right-hand side of (\ref{cfe}) represent the formation and disappearance of particles of volume $z$ respectively due to coagulation events. On the other hand, the third integral represents the birth of particles of volumes $z$ due to collision between pair of particles of volumes $z_1$ and $z_2$, in which the mass of particles undergoing breakage is larger than $z$ and the last integral describes the death of particles of volume $z$ due to collision between particles of volume $z$ and remaining particles of the system.\\

The breakup kernel has following properties:\\
$(i)$ The total number of particles resulting from the breakage of a single particle of volume $z_1$ after its collision with a particle of volume $z_2$ is given by
\begin{align}\label{N1}
\int_{0}^{z_1}B(z|z_1;z_2)dz = \zeta(z_1),\ \text{for all}\  z_1>0,\  B(z|z_1;z_2)=0\  \text{for}\  z> z_1,
\end{align}
 where $\zeta(z_1)$ represents the number of fragments obtained from the breakage of particles of volume $z_1 \in \mathbb{R}_{+}$, after collision with particles of volume $z_2$. Additionally, it is assumed that $\sup_{z_1 \in \mathbb{R}_{+}} \zeta(z_1) = N < \infty$, where $N \geq 2$.\\
$(ii)$ A necessary condition for mass conservation during collision-induced multiple fragmentation events is
\begin{align}\label{mass1}
\int_{0}^{z_1}zB(z|z_1;z_2)dz = z_1,\ \ \text{for all}\ \ z_1>0.
\end{align}
 From the condition (\ref{mass1}), the total volume $z_1$ of particles is conserved during the breakage of a particle of volume $z_1$ due to its collision with a particle of volume $z_2$. A more detailed study of the properties of the breakage function can be found in \cite{Kostoglou:2000}.\\

Moreover, it is important to define moments of concentration $g$. Let $M_r$ denotes the $r^{th}$ moment of the concentration $g(z,t)$, which is defined as
\begin{align*}
M_r(t)=M_r(g(z,t)) := \int_0^{\infty} z^r g(z,t)dz,\ \ \text{ where}\ \ r \geq 0.
\end{align*}
The zeroth and first moments represent the total number of particles and total mass of particles respectively. In collision-induced multiple fragmentation events, the total number of particles increases while in coagulation events, the total number of particles decreases. In addition, it is expected that the total mass of the system remains constant during these events. However, sometimes the mass conserving property breaks down due to high growth of kernels. Hence, either\emph{ gelation} or \emph{shattering transition} may occur in the system.\\

The present work mainly deals with the existence and uniqueness of  weak solutions to continuous coagulation equation with collision-induced multiple fragmentation, (\ref{cfe})--(\ref{in1}). At the end, it is also observed that the unique solution satisfies the mass conservation property. There are several mathematical results available on the existence and uniqueness of solutions to coagulation-fragmentation equations which are obtained by applying various techniques under different growth conditions on coagulation and fragmentation kernels, see \cite{Barik:2017, DaCosta:1995, Dubovskii:1996, Giri:2010A, Giri:2011, McLaughlin:1997, Stewart:1989, Stewart:1990}. However, best to our knowledge, there is no research article available dealing with the continuous coagulation equation with collision-induced multiple fragmentation. Nevertheless, in \cite{Jianhong:2008}, a discrete coagulation process with nonlinear binary fragmentation is considered, where authors have discussed an analytical solution to discrete coagulation equation with binary collision-induced fragmentation for constant coagulation kernel and volume dependent fragmentation kernel. In addition, there are a few articles, in which analytical solutions to nonlinear fragmentation equation have been investigated only for specific collision and breakup kernels, see \cite{Cheng:1990, Cheng:1988, Matthieu:2007, Kostoglou:2000}. In 1988, the nonlinear fragmentation model was first introduced by Cheng and Redner \cite{Cheng:1988}. In \cite{Cheng:1988}, authors have discussed the scaling form of cluster size distribution and asymptotic behaviour of solutions to the continuous nonlinear fragmentation equation. Moreover, they have shown the basic difference of the scaling solutions to both linear and nonlinear fragmentation equations by taking some specific homogeneous collision kernel such as $C(az,az_1)=a^{\lambda}C(z, z_1)$ and the homogeneous breakup kernel such as $B(az|az_1;az_2)=a^{-1}B(z|z_1;z_2)$. In 1990, Cheng and  Redner \cite{Cheng:1990} have proposed a specific class of splitting model for the nonlinear fragmentation equation in which a pair of particles collide to each other. As a result of this collision, both particles are splitting in different ways: $(i)$ in exactly two, $(ii)$ only the large one is splitting or $(iii)$ only the smaller one is splitting. They have also derived asymptotic behaviour of the scaling solution by using homogeneous collision kernel $C(z, z_1)=(zz_1)^{\frac{\lambda}{2}}$ and breakup kernel in different splitting model for nonlinear fragmentation equation. Later, Kostoglou and Karabelas \cite{Kostoglou:2000}, have discussed an analytical and asymptotic information of solution to the nonlinear fragmentation equation. They have considered different simple homogeneous collision and breakup kernels to transform the nonlinear fragmentation equation into linear one for discussing the self-similar solutions. Recently, Ernst and Pagonabarra \cite{Matthieu:2007} have inquired some more details about the scaling solutions and occurrence of shattering transition for different breakage models such as symmetric breakage, L-breakage and S-breakage of nonlinear fragmentation equation. Here symmetric breakage, L-breakage and S-breakage denote respectively, the splitting of both particles into exactly two pieces, splitting of the large particle only and the smaller particle only, see \cite{Cheng:1990}. In \cite{Laurencot:2001, Safronov:1972}, authors have discussed the coagulation and collisional breakage equation. In particular, in \cite{Laurencot:2001}, authors have solved the discrete nonlinear fragmentation equation mathematically by using weak $L^1$ compactness method. However, it is quite delicate to handle mathematically the continuous nonlinear fragmentation equation because small sized particles are fragmented into very small sized to form an infinite number of clusters in a finite time. In order to overcome this problem, we consider a fully nonlinear continuous coagulation-fragmentation model which is known as the \emph{continuous coagulation model with collision-induced multiple fragmentation} (\ref{cfe}).\\

 Best to our knowledge, this is the first attempt to show the existence and uniqueness of weak solutions to the continuous coagulation equations with collision-induced multiple fragmentation, (\ref{cfe})--(\ref{in1}) for large classes of unbounded coagulation, collision and breakup kernels.\\

The paper is arranged as follows: In Section 2, we state some definitions, hypotheses and lemmas, which are essentially required for upcoming  results in subsequent sections. In Section 3, we show the existence of weak solutions to continuous coagulation equations with collision-induced multiple fragmentation (\ref{cfe})--(\ref{in1}) by using a weak $L^1$ compactness method, which has been widely discussed for coagulation equation with linear fragmentation, see \cite{Giri:2010A, Giri:2012, McLaughlin:1997, Stewart:1989}. In Section 4, the uniqueness of weak solutions to (\ref{cfe})--(\ref{in1}) is shown which is motivated by Giri \cite{Giri:2013} and  Escobedo et al. \cite{Escobedo:2003}. The proof relies on the integrability of higher moments. In addition, the mass conservation property of the unique solution is also studied in this section.

\section{Some definitions and results}
In order to prove the existence and uniqueness of weak solutions to (\ref{cfe})--(\ref{in1}),
 define the following Banach space $S^+$ as
\begin{align*}
 S^+:=\{  g \in L^1(\mathbb{R}_{+}, dz): \|g\|_{L^1(\mathbb{R}_{+}, (1+z)dz)} < \infty\ \text{and}\ g \geq 0\ \text{a.e.} \},
 \end{align*}
where
\begin{align*}
\|g\|_{L^1(\mathbb{R}_{+}, (1+z)dz)}:=\int_{0}^{\infty}(1+z)|g(z)|dz,
\end{align*}
In another way, we also define the norms
\begin{align*}
\|g\|_{L^1(\mathbb{R}_{+}, zdz)}:=\int_{0}^{\infty}z|g(z)|dz
\end{align*}
and
\begin{align*}
\|g\|_{L^1(\mathbb{R}_{+}, dz)}:=\int_{0}^{\infty}|g(z)|dz,\ \text{where}\ g\in S^{+}.
\end{align*}

We show that the weak solutions of (\ref{cfe})--(\ref{in1}) lie in $S^+$. Now, we formulate weak solutions to the given nonlinear integro-partial differential equations (\ref{cfe})--(\ref{in1}) through the following definition:
\begin{defn}\label{def1} Let $T \in \mathbb{R}_{+}$. A solution $g$ of (\ref{cfe})--(\ref{in1}) is a non-negative function $g: [0,T]\to S^+$ such that, for a.e. $z\in \mathbb{R}_{+}$ and all $t\in [0,T]$,\\
$(i)$  $s\mapsto g(z,s)$ is continuous on $[0,T]$,\\
$(ii)$  the following integrals are finite
     \begin{align*}
     &\int_{0}^{t}\int_{0}^{\infty}K(z, z_1)g(z_1,s)dz_1ds<\infty,\ \ \int_{0}^{t}\int_{0}^{\infty}C(z, z_1)g(z_1,s)dz_1ds<\infty\nonumber\\
      & \mbox{and} \ \ \int_{0}^{t}\int_{z}^{\infty} \int_{0}^{\infty}B(z|z_1;z_2)C(z_1,z_2)g(z_1,s)g(z_2,s)dz_2dz_1ds<\infty,
     \end{align*}
$(iii)$  the function $g$ satisfies the following weak formulation of (\ref{cfe})--(\ref{in1})
\begin{align*}
g(z,t)=&g_0(z)+\frac{1}{2}\int_{0}^{t} \int_{0}^{z}K(z-z_1, z_1)g(z-z_1,s)g(z_1,s)dz_1ds -\int_{0}^{t} \int_{0}^{\infty}K(z,z_1)g(z,s)g(z_1,s)dz_1ds\nonumber\\
&+\int_{0}^{t}\int_{z}^{\infty} \int_{0}^{\infty} B(z|z_1;z_2)C(z_1,z_2)g(z_1,s)g(z_2,s)dz_2dz_1ds\nonumber\\
&-\int_{0}^{t}\int_{0}^{\infty}C(z, z_1)g(z,s)g(z_1,s)dz_1ds.
\end{align*}
\end{defn}

Next, we state some hypotheses under which the existence of weak solutions to (\ref{cfe})--(\ref{in1}) is established.
\begin{hypp}\label{hyp1}
$(H1)$ $K$ and $C$ are  non-negative measurable functions on $\mathbb{R}_{+}^2:=(0, \infty)\times (0, \infty)$,\\
\\
$(H2)$  both $K$ and $C$  are symmetric, i.e. $K(z, z_1)=K(z_1,z)$ and $C(z, z_1)=C(z_1,z)$ for all $(z, z_1) \in \mathbb{R}_{+}^2$, \\
\\
$(H3)$ $K(z, z_1) \leq k_1(1+z)^{\omega}(1+z_1)^{\omega}$  for all $(z, z_1)\in \mathbb{R}_{+}^2$, $0 \leq \omega  < 1$ and some constant $k_1 >0$,\\
\\
$(H4)$  $C(z, z_1) = k_2(z^{\alpha}{z_1}^{\beta}+{z_1}^{\alpha}z^{\beta})$ for all $(z, z_1)\in \mathbb{R}_{+}^2$, $0< \alpha \leq \beta <1$ and for some constant $k_2 \geq 0$. In addition,  $K$ and $C$ satisfy locally  the following condition:
\begin{align*}
K(z, z_1) \geq 2(\zeta(z_1) -1)C(z, z_1),\ \ \ \ \forall (z, z_1) \in (0, 1) \times (0, 1),
\end{align*}
where $\zeta(z_1)$ is given in (\ref{N1}),
\\
$(H5)$  for each $W>0$ and for $z_1 \in (0,W),$ $0 < \alpha \leq \beta < 1$ and any measurable subset $U$ of  $(0,1)$ with $|U| \leq \delta$, we have
\begin{align*}
\int_{0}^{z_1} \chi_{U}(z) ({z_1}^{\alpha}\vee {z_1}^{\beta}) B(z|z_1;z_2)dz \leq \Omega_1(|U|, W), \  \text{where}\  \lim_{\delta \to 0} \Omega_1 ( \delta, W )=0,
\end{align*}
where $|U|$ denotes the Lebesgue measure of $U$, $\chi_{U}$ is the characteristic function  of $U$ given by
\begin{align*}
\chi_{U}(z):=\begin{cases}
1,\ \ & \text{if}\ z\in U, \\
0,\ \ &  \text{if}\ z\notin U,
\end{cases}
\end{align*}
$(H6)$ for $z_1 > W,$ we have $B(z|z_1;z_2) \leq k(W)z^{-\tau_2} $ for $z\in (0,W)$, $z_1 \in \mathbb{R}_{+}$, where $\tau_2 \in [0,1)$ and $k(W)>0$.
\end{hypp}

\begin{rmkk} For $k_1 = 0$ in $(H3)$, (\ref{cfe}) is transformed into purely nonlinear multiple fragmentation equation. In this case, it is difficult to control the total number of particles due to the repeated breakage of small size particles which leads to obtain an infinitely many clusters in a finite time. Therefore, the existence of weak solutions to (\ref{cfe})--(\ref{in1}) can not be shown by using weak compactness argument in $S^+$.
\end{rmkk}

In order to show the uniqueness of weak solution, we need to consider the following hypotheses on the breakup and collision kernels:\\
$(UH1)$ There is a constant $B_a>0$ and $\alpha +\beta := 2(1+\eta) > 0$ such that
\begin{align*}
B(z|z_1;z_2)C(z_1,z_2) \geq B_a {z_1}^{\eta}{z_2}^{1+\eta}\ \text{for any}\ z_1\geq 1,~z_2 \in \mathbb{R}_{+}\ \text{and}\ z\in(0,z_1).
\end{align*}
The condition $(UH1)$ is called the \emph{strong nonlinear fragmentation}.\\

The mass conserving property of the  unique solution can also easily be verified by using the integrability of higher moments.\\
\begin{rmkk}
In order to show the existence of weak solutions to (\ref{cfe})--(\ref{in1}), we consider that the coagulation kernel in $(H4)$ is sufficient strong locally than the collision kernel in the range $(0, 1) \times (0, 1)$, whereas to prove uniqueness result we consider the fragmentation dominates the coagulation for sufficiently large particles (denoted as strong nonlinear fragmentation in $(UH1)$).
\end{rmkk}

 Let us take a few examples of coagulation and breakup kernels which satisfy hypotheses $(H1)$--$(H6)$. The examples of coagulation kernels are exactly the same which are considered in Giri et al. \cite{Giri:2010A}.

Let us now turn to the following type of breakup kernels
  \begin{align*}
  B(z|z_1;z_2)= &(\nu +2)\frac{z^{\nu}}{{z_1}^{\nu +1}},\ \text{where}\ \  -2 <\nu \leq 0\  \text{and}\  z<z_1.
 \end{align*}
Since this breakage function has a physical meaning only if $-2< \nu \leq 0$. For $\nu =0$, this gives the binary fragmentation and for $-1< \nu \leq 0$,
 we get the finite number of particles, which is denoted by $\zeta(z_1)$ and written as $\zeta(z_1) = \frac{\nu +2}{\nu +1} \leq N$. But, for  $-2< \nu < -1$, we obtain an infeasible number of particles and for the case of $\nu =-1$, we obtain an infinite number of particles. It is clear from (\ref{N1}).\\

 Now, hypothesis  $(H5)$ is checked in the following way: for $z_1\in (0,W)$ and $W>0$ is fixed,
\begin{align*}
\int_{0}^{z_1} \chi_{U}(z) ( {z_1}^{\alpha} \vee {z_1}^{\beta}) B(z|z_1;z_2)dz = (\nu +2)( {z_1}^{\alpha} \vee {z_1}^{\beta}) \int_{0}^{z_1} \chi_{U}(z) \frac{z^{\nu}}{{z_1}^{\nu +1}} dz.
\end{align*}
For $p>1$, applying H\"{o}lder's inequality, we get
\begin{align*}
(\nu +2)( {z_1}^{\alpha} \vee {z_1}^{\beta}){z_1}^{-\nu -1}&\int_{0}^{z_1} z^{\nu} \chi_{U}(z) dz \leq (\nu +2)( {z_1}^{\alpha} \vee {z_1}^{\beta}){z_1}^{-\nu -1}  |U|^{\frac{p-1}{p}}\bigg( \int_0^{z_1} z^{p\nu}dz \bigg)^{1/p}\nonumber\\
%=& (\nu +2)( {z_1}^{\alpha} \vee {z_1}^{\beta}){z_1}^{-\nu -1} |U|^{\frac{p-1}{p}}\bigg(\frac{z^{p\nu +1}}{p\nu +1}\bigg|_0^{z_1} \bigg)^{1/p}\nonumber\\
=& (\nu +2)( {z_1}^{\alpha} \vee {z_1}^{\beta}){z_1}^{-\nu -1} |U|^{\frac{p-1}{p}}\bigg(\frac{{z_1}^{p\nu +1}}{p\nu +1} \bigg)^{1/p}, \mbox{~for~} \nu > -1/p \nonumber\\
\leq &\frac { \nu +2}{(p \nu +1)^{1/p}} |U|^{\frac{p-1}{p}} ({{W_1}^{\alpha -1+1/p}}\vee{{W_1}^{ \beta  -1+1/p}}), \mbox{~for~ } \alpha  \geq 1-1/p.
%\leq &\frac {\nu +2}{(p \nu +1)^{1/p}} |U|^{\frac{p-1}{p}}{W^{(\alpha \vee \beta ) -1+1/p}}.
\end{align*}
 This implies that
\begin{align*}
\int_{0}^{z_1} \chi_{U}(z) ({z_1}^{\alpha}\vee {z_1}^{\beta} ) B(z|z_1;z_2)dz \leq \Omega_1 (|U|, W).
\end{align*}
It is worth mentioning that $p$ can be found only if $\nu >-1$ and $\alpha >0$.\\

 In order to verify the hypothesis $(H6),$ for $ {z_1} > W $ and $W>0$ is fixed, we have
\begin{align*}
 B(z|z_1;z_2)= (\nu +2)\frac{z^{\nu}}{{z_1}^{\nu +1}}\leq (\nu +2) \frac{z^{\nu}}{W ^{1+\nu}}\leq  k(W) z^{-\tau_2},
\end{align*}
where  $-1<\nu \leq 0$, $\tau_2 =-\nu \in [0,1)$  and  $k(W)\geq \frac{\nu +2}{W^{1+\nu}}$.\\

Let us verify the last hypothesis $(UH1)$ by considering the lower bound on collision kernel as follows:\\
For any ${z_1} \geq 1$, ${z_2} \in \mathbb{R}_{+}$ and $z\in (0,{z_1})$,
\begin{align*}
B(z|z_1;z_2)C({z_1},z_2) \geq & 2 ({z_1}z_2)^{\frac{\alpha+\beta}{2}} \bigg(\frac{\nu +2}{z_1}\bigg)  \bigg(\frac{z}{z_1}\bigg)^{\nu} \geq  2(\nu +2)  {z_1}^{\frac{\alpha+\beta}{2} -1} {z_2}^{\frac{\alpha+\beta}{2}} =  B_a  {z_1}^{\eta} {z_2}^{1+\eta},
\end{align*}
where $ \frac{\alpha+\beta}{2} :=1+\eta > 0$ and $B_a=2(\nu +2)$.

Now we are in the position to state the following existence result:
%%%%%%%%%%%%%%%%%%%%%%%%%%%%%%%%%%%%%%%%%%%%%%%%%%%%%%%%%%%%%%%%%%%%%%%%%%%
\begin{thm}\label{existmain theorem1}
Suppose that $(H1)$--$(H6)$ hold and assume that the initial value $g_0\in S^+$. Then, (\ref{cfe})--(\ref{in1}) have a weak solution $g \in S^+ $.
\end{thm}

\section{Existence}\label{existexistence1}
%%%%%%%%%%%%%%%%%%%%%%%%%%%%%%%%%%%%%%%%%%%%%%%%%%%%%%%%%%%%%%%%%%%
%%%%%%%%%%%%%%%%%%%%%%%%%%%%%%%%%%%%%%%%%%%%%%%%%%%%%%%%%%%%%%%%%%%
 In order to prove the existence of weak solutions to (\ref{cfe})--(\ref{in1}), we follow the weak $L^1$ compactness method introduced in the classical work of Stewart \cite{Stewart:1989}.

%%%%%%%%%%%%%%%%%%%%%%%%%%%%%%%%%%%%%%%%%%%%%%%%%%%%%%%%%%%%%%%%%%%
\subsection{The truncated continuous coagulation equation with collision-induced multiple fragmentation}\label{subsec trunc1}
%%%%%%%%%%%%%%%%%%%%%%%%%%%%%%%%%%%%%%%%%%%%%%%%%%%%%%%%%%%%%%%%%%%

To show the existence result, first we write (\ref{cfe})--(\ref{in1}) into the limit of a sequence of truncated equations obtained by changing the coagulation and collision kernels $K$ and $C$ by their cut-off kernels $K_n$ and $C_n$ respectively \cite{Stewart:1989}, where

%\hrule
\[
% \begin{equation*}
\hspace{-.7cm} K_n(z,z_1):=\begin{cases}
K(z, z_1),\  & \text{if}\ z+z_1 \leq n, \\
\text{0},\  &  \text{if}\ z+z_1 > n,
\end{cases}
%\end{equation*}
\ \text{and}\
%\begin{equation*}
C_n(z,z_1):=\begin{cases}
C(z, z_1),\  & \text{if}\ z+z_1 \leq n, \\
\text{0},\  &  \text{if}\ z+z_1 > n,
\end{cases}
%\end{equation*}
\]
%\hrule

for $n\ge 1$ and $n \in \mathbb{N}$.

For boundedness of $C_n$ and $K_n$  for each $n\ge 1$, we may follow as in \cite[Theorem 3.1]{Stewart:1989} or \cite{Walker:2002} to show the truncated equation
\begin{align}\label{trunc1}
\frac{\partial g^n(z,t)}{\partial t}  = &\frac{1}{2} \int_{0}^{z} K_n(z-z_1, z_1)g^n(z-z_1,t)g^n(z_1,t)dz_1- \int_{0}^{n-z} K_n(z, z_1)g^n(z,t)g^n(z_1,t)dz_1\nonumber\\
 & +\int_{z}^{n}\int_{0}^{n-z_1}B(z|z_1;z_2)C_n(z_1,z_2)g^n(z_1,t)g^n(z_2,t)dz_2dz_1\nonumber\\
 & - \int_{0}^{n-z} C_n(z, z_1)g^n(z,t)g^n(z_1,t)dz_1,
\end{align}
with given initial data
\begin{equation}\label{trunc in1}
g^{n}_0(z):=\begin{cases}
g_0(z),\ \ & \text{if}\ 0 < z< n, \\
\text{0},\ \ &  \text{if}\ z\geq n,
\end{cases}
\end{equation}
has a unique non-negative solution $g^n\in \mathcal{C}([0, T];L^1((0,n), dz))$ such that $g^n(z,t) \in S^+$ for all $t\ge 0$. Additionally, it satisfies the mass conservation property for all $t\in[0, T]$, i.e.
\begin{align}\label{trunc mass1}
\int_{0}^{n}zg^n(z,t)dz=\int_{0}^{n}zg^n_0(z)dz.
\end{align}

 In addition, we extend the truncated solution $g^n$ by zero in $\mathbb{R}_{+} \times \mathbb{R}_{+}$, as
\begin{equation}\label{trunc soln}
g^{n}(z, t):=\begin{cases}
g(z, t),\ \ & \text{if}\ 0 < z< n, \\
\text{0},\ \ &  \text{if}\ z\geq n,
\end{cases}
\end{equation}
for $n\ge 1$ and $n \in \mathbb{N}$.

Next, we wish to establish suitable bounds to apply Dunford-Pettis theorem [\cite{Edwards:1965}, Theorem 4.21.2] and then  equicontinuity of the sequence $(g^n)_{n\in \mathbb{N}}$ in time to use the \textit{Arzel\`{a}-Ascoli Theorem} \cite[Appendix A8.5]{Ash:1972}. This is the aim of the next section.

%%%%%%%%%%%%%%%%%%%%%%%%%%%%%%%%%%%%%%%%%%%%%%%%%%%%%%%%%%%%%%%%%%%
\subsection{ Weak compactness}\label{subs:wk}
%%%%%%%%%%%%%%%%%%%%%%%%%%%%%%%%%%%%%%%%%%%%%%%%%%%%%%%%%%%%%%%%%%%
\begin{lem}\label{compactness1}
Assume that $(H1)$--$(H6)$ hold and fix $T>0$. Let $g_0 \in S^+$ and $g^n$ be solution to (\ref{trunc1})--(\ref{trunc in1})  Then, the followings hold true: \\
$(i)$ there is a constant $V(T)>0$ (depending on $T$) such that
\begin{align*}
\int_{0}^{\infty}(1+z)g^n(z,t)dz\leq V(T)\ \ \text{for}\ \  n\ge 1\ \ \text{and all} \ \ t\in [0,T],
\end{align*}
$(ii)$ for any given $\epsilon> 0$, there exists $W_\epsilon>0$ (depending on $\epsilon$) such that, for all $t\in[0,T]$\\
\begin{align*}
\sup_{n\ge 1} \left\{ \int_{W_\epsilon}^{\infty}g^n(z,t)dz \right\}\leq \epsilon,
\end{align*}
$(iii)$ for a given $\epsilon > 0$, there exists  $\delta_\epsilon>0$ (depending on $\epsilon$) such that, for every measurable set $U$ of $\mathbb{R}_{+}$ with $|U|\leq \delta_\epsilon$, $n \ge 1$ and $t\in [0,T]$,\\
\begin{align*}
\int_{U}g^n(z,t)dz< \epsilon.
\end{align*}
\end{lem}

\begin{proof}
$(i)$  Let $n\geq 1$ and $t\in [0,T]$, where $T>0$ is fixed. For $n=1$, the proof is trivial. Next, for $n>1$ and then taking integration of (\ref{trunc1}) from $0$ to $1$  with respect to $z$ and by using Leibniz's rule, we obtain
\begin{align}\label{Unibound1}
\frac{d}{dt} \int_0^1  g^n(z,t)dz =&\frac{1}{2} \int_0^1\int_{0}^{z} K_n(z-z_1, z_1)g^n(z-z_1,t)g^n(z_1,t)dz_1dz\nonumber\\
&- \int_0^1\int_{0}^{n-z} K_n(z, z_1)g^n(z,t)g^n(z_1,t)dz_1dz\nonumber\\
&+\int_0^1 \int_z^{n}\int_0^{n-z_1} B(z|z_1;z_2)C_n(z_1,z_2)g^n(z_1,t)g^n(z_2,t)dz_2dz_1dz\nonumber\\
&-\int_0^1 \int_0^{n-z}C_n(z, z_1)g^n(z,t)g^n(z_1,t)dz_1dz.
\end{align}
The first term on the right-hand side of (\ref{Unibound1}) can be simplified by using Fubini's theorem and using $z-z_1={z_1}^{'}$ and $z=z^{'}$ as
\begin{align}\label{Unibound2}
\frac{1}{2} \int_0^1\int_{0}^{z} K_n(z-z_1, z_1)g^n(z-z_1,t)g^n(z_1,t)dz_1dz =\frac{1}{2} \int_0^1\int_{0}^{1-z_1} K_n(z, z_1)g^n(z,t)g^n(z_1,t)dzdz_1.
\end{align}
Using Fubini's theorem, the third term of (\ref{Unibound1}) can be written as
\begin{align}\label{Unibound3}
\int_0^1  \int_z^{n}\int_0^{n-z_1}& B(z|z_1;z_2)C_n(z_1,z_2)g^n(z_1,t)g^n(z_2,t)dz_2dz_1dz\nonumber\\
\leq & \int_0^1\int_0^{n-z_1} \zeta(z_1) C_n(z_1,z_2)g^n(z_1,t)g^n(z_2,t)dz_2dz_1\nonumber\\
&+\int_1^n \int_0^{1}\int_0^{n-z_1} B(z|z_1;z_2)C_n(z_1,z_2)g^n(z_1,t)g^n(z_2,t)dz_2dzdz_1.
\end{align}
Substituting (\ref{Unibound2}) and (\ref{Unibound3}) into (\ref{Unibound1}) and then using $(H4)$ and (\ref{N1}), we obtain
\begin{align}\label{Unibound4}
\frac{d}{dt} \int_0^1  g^n(z,t)dz
\leq &-\frac{1}{2} \int_0^1\int_{0}^{1-z} [K_n(z, z_1)-2(\zeta(z_1)-1)C_n(z, z_1)]g^n(z,t)g^n(z_1,t)dz_1dz\nonumber\\
&- \int_0^1\int_{1-z}^{1} [K_n(z, z_1)-(\zeta(z_1)-1)C_n(z, z_1)]g^n(z,t)g^n(z_1,t)dz_1dz\nonumber\\
&- \int_0^1\int_{1}^{n-z} K_n(z, z_1)g^n(z,t)g^n(z_1,t)dz_1dz\nonumber\\
&+\int_1^n \int_0^{1}\int_0^{n-z_1} B(z|z_1;z_2)C_n(z_1,z_2)g^n(z_1,t)g^n(z_2,t)dz_2dzdz_1\nonumber\\
&+(N-1)\int_0^1\int_{1}^{n-z} C_n(z, z_1)g^n(z,t)g^n(z_1,t)dz_1dz\nonumber\\
\leq & Nk_2\int_1^n \int_0^{n-z_1} z_1(z_2^{\alpha}+z_2^{\beta})g^n(z_1,t)g^n(z_2,t)dz_2dz_1\nonumber\\
&+Nk_2\int_0^1\int_{1}^{n-z} z_1(z^{\alpha}+z^{\beta})g^n(z,t)g^n(z_1,t)dz_1dz\nonumber\\
\leq & 2Nk_2 \|g_0\|_{L^1(\mathbb{R}_{+}, zdz)}  \bigg[ 2\int_0^1 g^n(z,t)dz+ \|g_0\|_{L^1(\mathbb{R}_{+}, zdz)}  \bigg].
\end{align}
Thanks to (\ref{trunc in1}) and $g_0\in S^+$.
Again, taking integration of (\ref{Unibound4}) from $0$ to $t$ with respect to time and then applying Gronwall's inequality, we have
\begin{align}\label{Unibound5}
\int_0^1 g^n(z,t)dz \leq V_1(T),
\end{align}
where
\begin{align*}
V_1(T):=\| g_0 \|_{L^1(\mathbb{R}_{+}, dz)} e^{4Nk_2T \| g_0 \|_{L^1(\mathbb{R}_{+}, zdz)}}+ \frac{\| g_0 \|_{L^1(\mathbb{R}_{+}, zdz)}}{2}\bigg( e^{4Nk_2T \| g_0 \|_{L^1(\mathbb{R}_{+}, zdz)}} -1 \bigg).
\end{align*}
Now, using (\ref{Unibound5}), (\ref{trunc mass1}) and (\ref{trunc in1}), estimate the following integral as
\begin{align*}
\int_0^n(1+z) g^n(z,t)dz =&\int_0^1 g^n(z,t)dz+\int_1^n g^n(z,t)dz+\int_0^n z g^n(z,t)dz\nonumber\\
\leq & \int_0^1 g^n(z,t)dz+2\| g_0 \|_{L^1(\mathbb{R}_{+}, zdz)} \leq V(T),
\end{align*}
where $V(T):=V_1(T)+2 \| g_0\|_{L^1(\mathbb{R}_{+}, zdz)} $. This completes the proof of the first part of Lemma \ref{compactness1}.

\bigskip

$(ii)$ For proving the second part of Lemma \ref{compactness1}, see Giri et al. \cite{Giri:2010A}.

\bigskip

$(iii)$ Choose $ \epsilon > 0$ and let $U \subset \mathbb{R}_{+}$. Using Lemma \ref{compactness1} $(ii)$, we can choose $W \in(0, n)$ such that for all $n\in \mathbb{N}$  and $ t \in [0,T]$,
\begin{align}\label{comp3 2}
\int_{W}^{\infty}g^n(z,t)dz <  \frac {\epsilon} {2}.
\end{align}
 Fix $W>0$, for $n\ge 1$, $\delta \in (0,1)$ and $t\in [0,T]$, we define
\begin{align*}
 r^n(\delta,t):=\sup \left\{\int_{0}^{W} \chi_{U}(z)g^n(z,t)dz\ :\ U\subset (0,W) \ \ \text{and} \ \
|U|\leq \delta \right\}.
\end{align*}
For $n\ge 1$ and $t\in [0,T]$, it follows from the non-negativity of $g^n$, Fubini's theorem and (\ref{trunc1})--(\ref{trunc in1}) that
\begin{align}\label{comp3 1}
\int_{0}^{W} \frac{\partial}{\partial t}  \chi_{U}(z)& g^n(z,t)dz %\leq &\frac{1}{2}\int_{0}^{W}\int_{0}^{W-x}\chi_{U(x+y)} K(z, z_1)g^n(z,t)g^n(z_1,t)dz_1dz\nonumber\\
%&+\int_{0}^{W}\chi_{U}(x) \int_{x}^{n}\int_{0}^{n-y}B(z|z_1;z_2)C(z_1,z_2)g^n(z_1,t)g^n(z_2,t)dzdz_1dz \nonumber\\
  \leq  \frac{1}{2}\int_{0}^{W}\int_{0}^{W-z}\chi_{U}(z+z_1)K_n(z, z_1)g^n(z,t)g^n(z_1,t)dz_1dz\nonumber\\
 % &+\int_{0}^{1}\int_{0}^{z_1}\int_{0}^{n-z_1}\chi_{U}(z) B(z|z_1;z_2)C(z_1,z_2)g^n(z_1,t)g^n(z_2,t)dz_2dzdz_1\nonumber\\
  & +\int_{0}^{W}\int_{0}^{z_1}\int_{0}^{n-z_1}\chi_{U}(z) B(z|z_1;z_2)C_n(z_1,z_2)g^n(z_1,t)g^n(z_2,t)dz_2dzdz_1\nonumber\\
   & +\int_{W}^{n}\int_{0}^{W}\int_{0}^{n-z_1}\chi_{U}(z) B(z|z_1;z_2)C_n(z_1,z_2)g^n(z_1,t)g^n(z_2,t)dz_2dzdz_1.
 \end{align}
Let us denote the first, second and third integrals on the right-hand side to (\ref{comp3 1}) by $I_1$, $I_2$ and $I_3$  respectively. Then, we estimate $I_1$, $I_2$ and $I_3$ separately.\\
 $I_1$ can be estimated similar to Giri et al. \cite{Giri:2012} as
\begin{align*}
I_1 \leq k_1 V(T)(1+W)r^n (\delta , t).
\end{align*}
  By using $(H4)$, $(H5)$ and  Fubini's theorem, the integral term $I_2$ is evaluated as
\begin{align*}
I_{2} = &\int_{0}^{W}\int_{0}^{z_1}\int_{0}^{n-z_1}\chi_{U}(z) B(z|z_1;z_2)C_n(z_1,z_2)g^n(z_1,t)g^n(z_2,t)dz_2dzdz_1\nonumber\\
  % &\leq  k_2\int_{0}^{W}\int_{0}^{y} \int_{0}^{n} \chi_{U}(z)B(z|z_1;z_2)({z_1}^{\alpha}{z_2}^{\beta}+{z_1}^{\beta}{z_2}^{\alpha})g^n(z_1,t)g^n(z_2,t)dz_2dzdz_1\nonumber\\
   \leq & k_2\int_{0}^{W}\int_{0}^{z_1} \int_{0}^{n} \chi_{U}(z) B(z|z_1;z_2)({z_1}^{\alpha}\vee {z_1}^{\beta}) ({z_2}^{\beta}+{z_2}^{\alpha}) g^n(z_1,t)g^n(z_2,t)dz_2dzdz_1\nonumber\\
  % &\leq  \Omega_1(|U|, W)k_2\int_{0}^{W}\int_{0}^{n}[(1+z)^{\beta}+(1+z)^{\alpha}]g^n(z_1,t)g^n(z_2,t)dzdy\nonumber\\
  % &\leq 2\Omega_1(|U|,W)k_{2}\int_{0}^{W}\int_{0}^{n}(1+z)g^n(z_1,t)g^n(z_2,t)dzdy\nonumber\\
   \leq  & 2k_2V(T)^2\Omega_1(|U|, W).
\end{align*}
 Similarly, from $(H4)$, $(H6)$, Fubini's theorem and Lemma \ref{compactness1} $(i)$, the third integral $I_3$ appeared on the right-hand side of (\ref{comp3 1}), can be estimated as \begin{align*}
   I_{3}&= \int_{W}^{n}\int_{0}^{W}\int_{0}^{n-z_1}\chi_{U}(z) B(z|z_1;z_2)C_n(z_1, z_2)g^n(z_1,t)g^n(z_2,t)dz_2dzdz_1\\
   & \leq  2k_{2} k(W)V(T)^2\int_{0}^{W}\chi_{U}(z)z^{-\tau_2}dz.
 \end{align*}
Now, applying  H\"older's inequality, we have
\begin{align*}
I_{3}
 \leq 2k_{2} k(W)V(T)^{2}  {\delta}^{\frac{1-\tau_2}{1+\tau_2}}  \left(  \frac{W^{\frac{1-\tau_2}{2}}}   {{\frac{1-\tau_2}{2}}} \right)^{\frac{2\tau_2}{1+\tau_2}}.
\end{align*}
Gathering the above estimates on $I_{1}$, $I_{2},$  $I_{3}$ and inserting them into (\ref{comp3 1}), we obtain
\begin{align*}
\frac{d}{dt}\int_{0}^{W} \chi_{U}(z)g^n(z,t)dz\leq &k_1 V(T)(1+W)r^n (\delta , t)+2k_{2}V(T)^2\Omega_1(|U|, W) \nonumber\\
%+2k_{1}V(T)^2\Omega(|U|,W)\nonumber\\
&+ 2k_{2} k(W)V(T)^{2}  {\delta}^{\frac{1-\tau_2}{1+\tau_2}}  \left(  \frac{W^{\frac{1-\tau_2}{2}}}   {{\frac{1-\tau_2}{2}}} \right)^{\frac{2\tau_2}{1+\tau_2}}.
\end{align*}

Integrating the above inequality with respect to $t$ and taking supremum over all $U$ such that $U \subset (0,W)$ with $|U|$ $\leq \delta$, we estimate
\begin{align*}
r^n(\delta,t) \leq & r^n(\delta,0)+k_1 V(T)(1+W)\int_0^tr^n (\delta , s)ds+2k_{2}V(T)^2T\Omega_1(|U|, W) \nonumber\\
&+ 2k_{2} k(W)TV(T)^{2}  {\delta}^{\frac{1-\tau_2}{1+\tau_2}}  \left(  \frac{W^{\frac{1-\tau_2}{2}}}   {{\frac{1-\tau_2}{2}}} \right)^{\frac{2\tau_2}{1+\tau_2}},\ \ t\in [0,T].
\end{align*}
An application of Gronwall's inequality finally gives
 \begin{align*}
r^n(\delta,t) \leq C^{*} (\delta, W) \exp(k_1V(T)T(1+W)),\ \ t\in [0,T],
\end{align*}
where
 \begin{align*}
 C^{*} (\delta, W):= & r^n(\delta,0)+2\Omega_1(|U|, W)k_{2}TV(T)^2 + 2k_{2} k(W)TV(T)^{2}  {\delta}^{\frac{1-\tau_2}{1+\tau_2}}  \left(  \frac{W^{\frac{1-\tau_2}{2}}}   {{\frac{1-\tau_2}{2}}} \right)^{\frac{2\tau_2}{1+\tau_2}}.
 \end{align*}
This shows that
\begin{align}\label{comp3 3}
\mbox{sup}_{n}\{ r^n(\delta,t)\}\rightarrow 0~~ \mbox{as}~~ \delta \rightarrow 0.
\end{align}
Adding (\ref{comp3 2}) and (\ref{comp3 3}), we thus obtain the required result.
\end{proof}
Hence, from Dunford-Pettis theorem, we have $(g^{n})_{n \in \mathbb{N}}$ is a relatively compact subset of $S^+$ for each $t \in [0, T]$.\\

Next, the equicontinuity in time of the family $\{g^n(t), t\in[0,T]\} $ in $L^1(\mathbb{R}_{+}, dz)$ can easily be shown as similar to \cite{Giri:2010A, Giri:2012, Stewart:1989} . Then according to a refined version of the \textit{Arzel\`{a}-Ascoli theorem}, see \cite[Theorem 2.1]{Stewart:1989} or \textit{Arzel\`{a}-Ascoli theorem}, see Ash [\cite{Ash:1972}, page 228], we conclude that there exists a subsequence (${g^{n_k}}$) such that
\begin{align*}
\lim_{n_k\to\infty} \sup_{t\in [0,T]}{\left\{ \left| \int_0^\infty  [ g^{n_k}(z,t) - g(z,t)]\ \phi(z)\ dz \right| \right\}} = 0, \label{vittel}
\end{align*}
for all $T>0$, $\phi \in L^\infty(\mathbb{R}_{+})$ and some $g \in \mathcal{C}_w([0,T]; L^1(\mathbb{R}_{+}, dz))$, where  $\mathcal{C}_w([0, T]; L^1 (\mathbb{R}_{+}, dz))$ is the space of all weakly continuous functions from $[0, T]$ to $L^1 (\mathbb{R}_{+},  dz)$. This implies that
\begin{equation}\label{equicontinuity f}
  g^{n_k}(t) \rightharpoonup g(t)\ \text{in}\  L^1(\mathbb{R}_{+}, dz) \ \text{as}\ n \to \infty,
\end{equation}
converges uniformly for $t \in [0,T]$ to some $g \in \mathcal{C}_w([0,T]; L^1(\mathbb{R}_{+}, dz))$.\\

Next, for any $ m>0$, $t\in [0,T], $ since we have $g^{n_k} \rightharpoonup g$, we obtain
\begin{align*}
\int_{0}^{m}zg(z,t)dz = \lim_{ n_k\rightarrow {\infty}} \int_{0}^{m} z g^{n_k}{(z,t)}dz \leq \|g_0\|_{L^1(\mathbb{R}_{+}, zdz)} < \infty.
\end{align*}
Using (\ref{trunc mass1}), the non-negativity of each $g^{n_k}$ and $g$, then as $ m\rightarrow \infty $ implies that $g \in S^+ $.

\subsection{Convergence of approximated integrals}\label{subs:limit}
%%%%%%%%%%%%%%%%%%%%%%%%%%%%%%%%%%%%%%%%%%%%%%%%%%%%%%%%%%%%%%%%%%%
Now we prove that the limit function $g$ obtained in (\ref{equicontinuity f}) is actually a weak solution to (\ref{cfe})--(\ref{in1}).
We shall use weak continuity and convergence properties of some operators which we define below. For $g\in S^+$, $n\ge 1$ and $z\in \mathbb{R}_{+}$, we define
\begin{align*}
 & P_{1}^n(g^n)(z,t):=\frac{1}{2}\int_{0}^{z}K_n(z-z_1, z_1)g^n(z-z_1,t)g^n(z_1,t)dz_1,\\
 & P_{1}(g)(z,t):=\frac{1}{2}\int_{0}^{z}K(z-z_1, z_1)g(z-z_1,t)g(z_1,t)dz_1,\\
 &P_{2}^n(g^n)(z,t):=\int_{0}^{n-z}K_n(z, z_1)g^n(z,t)g^n(z_1,t)dz_1,\  \ P_{2}(g)(z,t):=\int_{0}^{\infty}K(z, z_1)g(z,t)g(z_1,t)dz_1,\\
 & P_{3}^n(g^n)(z,t):=\int_{z}^{n}\int_{0}^{n-z_1}B(z|z_1;z_2)C_n(z_1,z_2)g^n(z_1,t)g^n(z_2,t)dz_2dz_1,\\
 & P_{3}(g)(z,t):=\int_{z}^{\infty}\int_{0}^{\infty}B(z|z_1;z_2)C(z_1,z_2)g(z_1,t)g(z_2,t)dz_2dz_1,\\
 & P_{4}^n(g^n)(z,t):=\int_{0}^{n-z}C_n(z, z_1)g^n(z,t)g^n(z_1,t)dz_1,\  \ P_{4}(g)(z,t):=\int_{0}^{\infty}C(z, z_1)g(z,t)g(z_1,t)dz_1,
\end{align*}
where $P^n:=P_1^n-P_2^n+P_3^n-P_4^n$ and $P:=P_1-P_2+P_3-P_4$.\\

We then have the following result:
%%%%%%%%%%%%%%%%%%%%%%%%%%%%%%%%%%%%%%%%%%%%%%%%%%%%%%%%%%%%%%%%%%%
\begin{lem}\label{convergence lemma1}
Let $(g^{n})_{n\in \mathbb{N}}$ be a bounded sequence in $S^+$ and $g\in S^{+}$, where $\|g^n\|_{L^1( \mathbb{R}_{+}, (1+z)dz)} \leq V(T)$ and $g^n\rightharpoonup g$ in $L^1(\mathbb{R}_{+}, dz)$ as $n\to \infty $. Then, for each $W> 0$ and $i\in {1,2, 3, 4}$, we have
\begin{equation}
P^n(g^n)\rightharpoonup P(g)\ \ \text{in} \ \ L^1((0,W), dz)\ \ \text{as}\ \ n\to \infty. \label{luchon}
\end{equation}
\end{lem}
%%%%%%%%%%%%%%%%%%%%%%%%%%%%%%%%%%%%%%%%%%%%%%%%%%%%%%%%%%%%%%%%%%%

\begin{proof} Let $W>0, z\in (0,W]$  and $\chi $ be the characteristic function.  Suppose $\phi$ belongs to $L^\infty(0,W)$. We prove that $P_i^n(g^n)\rightharpoonup P_i(g)$ for $i=1, 2, 3, 4.$\\

For $i=1,2$, $P_i \rightharpoonup P_i$ can easily be shown as in \cite{Giri:2010, Giri:2012, Stewart:1989}.\\

For $i=3$\\
Given $\epsilon >0$ and we can choose $b> W$ large enough such that
\begin{align}\label{conv21}
\frac{2k_2 k(W) W^{1-\tau_2}} {1-\tau_2} \| \phi \|_{L^{\infty}(0,W)}[  V(T)^2+\|g\|^2_{L^1(\mathbb{R}_{+}, (1+z)dz)} ](1+b)^{\beta -1}< \frac{\epsilon}{2}.
\end{align}
Then, by $(H4)$, $(H6)$, (\ref{conv21}) and Fubini's theorem, we obtain
\begin{align}\label{conv22}
\bigg| \int_{0}^{W}\int_{b}^{\infty}\int_{0}^{\infty}& \phi(z)B(z|z_1;z_2)C(z_1,z_2)[g^n(z_1,t)g^n(z_2,t)-g(z_1,t)g(z_2,t)]dz_2dz_1dz \bigg| \nonumber\\
\leq &2k_2V(T) k(W) \| \phi \|_{L^{\infty}(0,W)} \int_{0}^{W}\int_{b}^{\infty} z^{-\tau_2} (1+z_1)^{\beta}g^n(z_1,t)dz_1dz\nonumber\\
&+2k_2k(W) \|g\|_{L^1(\mathbb{R}_{+}, (1+z)dz)} \| \phi \|_{L^{\infty}(0,W)} \int_{0}^{W}\int_{b}^{\infty} z^{-\tau_2} (1+z_1)^{\beta}g(z_1,t)dz_1dz\nonumber\\
%\leq & 2k_2 \| \phi \|_{L^{\infty}(0,W)} N\int_{b}^{\infty}\frac{(1+y)}{(1+y)^{1-\beta }}(g^n(z_1,t)V(T)+g(z_1,t)\|g\|)dz_1dz\nonumber\\
\leq & \frac{2k_2 k(W) W^{1-\tau_2}} {1-\tau_2} \| \phi \|_{L^{\infty}(0,W)} [V(T)^2+\|g\|^2_{L^1(\mathbb{R}_{+}, (1+z)dz)}](1+b)^{\beta - 1} < \frac{\epsilon}{2}.
\end{align}
Again for $\epsilon >0$ and we can find sufficiently large $c >0$  such that
\begin{align}\label{conv24}
2k_2 N\| \phi \|_{L^{\infty}(0,W)}[  V(T)^2+\|g\|^2_{L^1(\mathbb{R}_{+}, (1+z)dz)} ](1+c)^{\beta -1}< \frac{\epsilon}{2}.
\end{align}
For a.e. $z\in(0,W]$, by using $(H3)$, $(H4)$, (\ref{N1}), (\ref{conv24}) and Lemma  \ref{compactness1} (i), we obtain
\begin{align}\label{conv23}
\bigg| \int_0^W \int_z^b \int_c^{\infty}\phi(z)&B(z|z_1;z_2)C(z_1,z_2)[g^n(z_1,t)g^n(z_2,t)-g(z_1,t)g(z_2,t)]dz_2dz_1dz\bigg| \nonumber\\
%\leq & 2k_2  \int_0^W \int_z^b \int_c^{\infty}\phi(z)B(z|z_1;z_2)(1+z_1)^{\beta}(1+z_2)^{\beta}[g^n(z_1,t)g^n(z_2,t)-g(z_1,t)g(z_2,t)]dz_2dz_1dz\nonumber\\
%\leq & 2k_2\|\phi\|_{L^{\infty}(0,W)}N   \int_c^{\infty}(1+z_2)^{\beta}[g^n(z_2,t)V(T)+g(z_2,t)\|g\|]dz_2\nonumber\\
%\leq & 2k_2\|\phi\|_{L^{\infty}(0,W)}N   \int_c^{\infty}\frac{(1+z)}{(1+z)^{1-\beta}}[g^n(z_2,t)V(T)+g(z,t)\|g\|]dz\nonumber\\
\leq & 2k_2N\|\phi\|_{L^{\infty}(0,W)}  (1+c)^{\beta -1} [V(T)^2+\|g\|^2_{L^1(\mathbb{R}_{+}, (1+z)dz)}]< \frac{\epsilon}{2}.
\end{align}

For $z_1>n$, from (\ref{trunc mass1}), we have $g^n(z_1, t)=0$. Thus the following integral is
 \begin{align}\label{pass 01}
\bigg| \int_{0}^{W}\int_{n}^{\infty}\int_{0}^{\infty}\phi(z)B(z|z_1;z_2)C(z_1,z_2)g^n(z_1,t)g^n(z_2,t)dz_2dz_1dz \bigg| = 0.
\end{align}
By using  $(H4)$, $(H6)$, (\ref{N1}) and Fubini's theorem, we simplify the following integral
\begin{align*}
 \int_{0}^{W}\int_{z}^{n}\int_{n-z_1}^{\infty}&\phi(z)B(z|z_1;z_2)C(z_1,z_2)g^n(z_1,t)g^n(z_2,t)dz_2dz_1dz \\
  =&\int_{0}^{W}\int_{0}^{z_1}\int_{n-z_1}^{\infty}\phi(z)B(z|z_1;z_2)C(z_1,z_2)g^n(z_1,t)g^n(z_2,t)dz_2dzdz_1 \\
  &+\int_{W}^{n}\int_{0}^{W}\int_{n-z_1}^{\infty}\phi(z)B(z|z_1;z_2)C(z_1,z_2)g^n(z_1,t)g^n(z_2,t)dz_2dzdz_1 \\
 \leq &  2\|\phi\|_{L^{\infty}(0,W)}k_2 \bigg[ V(T)N \int_{n-z_1}^{\infty} (1+z_2)^{\beta}g^n(z_2,t)dz_2 \\ &
 + k(W)\int_{W}^{n}\int_{0}^{W} \int_{n-z_1}^{\infty}  z^{-\tau_2}(1+z_1)^{\beta}g^n(z_1,t)  (1+z_2)^{\beta}g^n(z_2,t)dz_2 dzdz_1 \bigg]\nonumber\\
 %\leq &  2\|\phi\|_{L^{\infty}(0,W)}k_2 V(T) \bigg[ NV(T)+\int_{W}^{n}\int_{0}^{W} k(W)x^{-\tau_2}(1+y)^{\beta}g^n(z_1,t)dzdz_1 \bigg]\nonumber\\
  \leq &  2k_2\|\phi\|_{L^{\infty}(0,W)} \bigg[ \frac{NV(T)^2}{(1+n_k-z_1)^{1-\beta}} +k(W)V(T)\frac{W^{1-\tau_2}}{1-\tau_2} \int_{n_k-z_1}^{\infty}  (1+z_2)^{\beta}g^n(z_2, t)dz_2 \bigg]\nonumber\\
\leq &  2k_2\|\phi\|_{L^{\infty}(0,W)} V(T)^2 \frac{\bigg[ N+k(W)\frac{W^{1-\tau_2}}{1-\tau_2}\bigg]}{(1+n_k-z_1)^{1-\beta} }.
 \end{align*}
 This implies
 \begin{align}\label{conv26}
   \int_{0}^{W}\int_{z}^{n}\int_{n-z_1}^{\infty}\phi(z)B(z|z_1;z_2)C(z_1,z_2)g^n(z_1,t)g^n(z_2,t)dz_2dz_1dz \to 0 \ \text{as}\ n_k \to \infty.
\end{align}

Next, let us consider the following integral by using Fubini's theorem, as
\begin{align}\label{conv27}
 \bigg| \int_0^W&\int_z^b \int_0^c  \phi(z)B(z|z_1;z_2)C(z_1,z_2)  g^n(z_1,t)g^n(z_2,t)-g(z_1,t)g(z_2,t) dz_2dz_1dz \bigg| \nonumber\\
%=& \int_0^W\int_z^b \int_0^c  \phi(z)B(z|z_1;z_2)C(z_1,z_2)  g^n(z_2, t)[g^n(z_1, t)-g(z_1, t)] dz_2dz_1dz \nonumber\\
%&+\int_0^W\int_z^b \int_0^c  \phi(z)B(z|z_1;z_2)C(z_1,z_2)  g(z_1, t)[g^n(z_2, t)-g(z_2, t)] dz_2dz_1dz\nonumber\\
\leq & \bigg|\int_0^W\int_0^{z_1} \int_0^c  \phi(z)B(z|z_1;z_2)C(z_1,z_2)  g^n(z_2, t)[g^n(z_1, t)-g(z_1, t)] dz_2dzdz_1\bigg| \nonumber\\
& +\bigg| \int_W^b \int_0^W \int_0^c  \phi(z)B(z|z_1;z_2)C(z_1,z_2)  g^n(z_2, t)[g^n(z_1, t)-g(z_1, t)] dz_2dzdz_1\bigg| \nonumber\\
&+\bigg| \int_0^W\int_0^{z_1} \int_0^c  \phi(z)B(z|z_1;z_2)C(z_1,z_2)  g(z_1, t)[g^n(z_2, t)-g(z_2, t)] dz_2dzdz_1 \bigg| \nonumber\\
&+\bigg| \int_W^b \int_0^W \int_0^c  \phi(z)B(z|z_1;z_2)C(z_1,z_2)  g(z_1, t)[g^n(z_2, t)-g(z_2, t)] dz_2dzdz_1\bigg| =: \sum_{l=1}^4 |Q_l^n|.
\end{align}
First, we show that $\lim_{n \to \infty} |Q_1^n|=0$. Since, we have $g \rightharpoonup g $ in $L^1(\mathbb{R_+}, dz)$, then
\begin{align*}
|Q_1^n|=& \bigg| \int_0^W  [g^n(z_1, t)-g(z_1, t)] \int_0^{z_1} \int_0^c  \phi(z)B(z|z_1;z_2)C(z_1,z_2)  g^n(z_2, t) dz_2dzdz_1 \bigg|.
\end{align*}
Next, by using $(H3)$ and $(\ref{N1})$, we have
\begin{align}\label{Q1}
  \int_0^{z_1} \int_0^c  \phi(z)B(z|z_1;z_2)C(z_1,z_2)  g^n(z_2, t) dz_2dz \leq 2\|\phi\|_{L^{\infty}(0,W)} N k_2 (1+W) V(T) \in L^{\infty} (0, W).
\end{align}
Thus, we get $\lim_{n \to \infty} |Q_1^n|=0$. Now, let us consider $|Q_2^n|$, as

\begin{align*}
|Q_2^n|=& \bigg| \int_W^b  [g^n(z_1, t)-g(z_1, t)] \int_0^{W} \int_0^c  \phi(z)B(z|z_1;z_2)C(z_1,z_2)  g^n(z_2, t) dz_2dzdz_1 \bigg|.
\end{align*}
Similarly, by using $(H3)$ and $(H6)$, we show that
\begin{align*}
 \int_0^{W} \int_0^c & \phi(z)B(z|z_1;z_2)C(z_1,z_2)  g^n(z_2, t) dz_2dz\nonumber\\
\leq & 2\|\phi\|_{L^{\infty}(0,W)}k_2 (1+z_1)V(T)k(W) \frac{W^{1-\tau_2}}{1-\tau_2} \in L^{\infty} (W, b).
\end{align*}
Then by weak convergence of $g^n$ to $g$ guarantees that
\begin{align}\label{Q2}
\lim_{n \to \infty } |Q_2^n|=0.
\end{align}
Using similar argument, one can easily be seen that $|Q_3^n|$ and $|Q_4^n|$ go to $0$, as $n \to \infty$.
Hence, we have
\begin{eqnarray}\label{conv28}
\bigg| \int_0^W\int_z^b \int_0^c  \phi(z)B(z|z_1;z_2)C(z_1,z_2)  g^n(z_1,t)g^n(z_2,t)-g(z_1,t)g(z_2,t) dz_2dz_1dz \bigg|  \to 0\ \ \text{as}\ n\to \infty.
\end{eqnarray}

Now, using (\ref{conv22}), (\ref{conv23}), (\ref{pass 01}), (\ref{conv26}) and (\ref{conv28}), we obtain for $n > W$
\begin{align}\label{pass 1}
\bigg| \int_{0}^{W}\phi(z) &[P_3(g^n)(z,t)-P_3(g)(z,t)] dz\bigg|\nonumber\\ %=&  \left| \int_{0}^{W}\int_{x}^{\infty} \int_{0}^{\infty}  \phi(z)B(z|z_1;z_2)C(z_1,z_2) [g^n(z_1,t)g^n(z_2,t)-g(z_1,t)g(z,t)]dzdz_1dz \right|\nonumber\\
 \leq &\bigg|\int_{0}^{W}\phi(z) \bigg[\int_{z}^{b}\int_{0}^c B(z|z_1;z_2)C(z_1,z_2)[g^n(z_1,t)g^n(z_2,t)-g(z_1,t)g(z_2,t)]dz_2dz_1  \nonumber\\
 &+ \int_{z}^{b}\int_{c}^{\infty} B(z|z_1;z_2)C(z_1,z_2)[g^n(z_1,t)g^n(z_2,t)-g(z_1,t)g(z_2,t)]dz_2dz_1  \nonumber\\
 &+ \int_{b}^{\infty}\int_{0}^{\infty} B(z|z_1;z_2)C(z_1,z_2)[g^n(z_1,t)g^n(z_2,t)-g(z_1,t)g(z_2,t)]dz_2dz_1 \nonumber \\
 &- \int_{n}^{\infty}\int_{0}^{\infty} B(z|z_1;z_2)C(z_1,z_2) g^n(z_1,t)g^n(z_2,t) dz_2dz_1 \nonumber\\
 &- \int_{z}^{n}\int_{n-z_1}^{\infty} B(z|z_1;z_2)C(z_1,z_2) g^n(z_1,t)g^n(z_2,t) dz_2dz_1\bigg]dz \bigg| < \epsilon.
 \end{align}
Since $\phi$ is arbitrary,
\begin{align*}
 \lim_{n \rightarrow \infty }{ P_3{(g^n)}\rightharpoonup P_3{(g)} }.
\end{align*}

Next, we show for $i=4$\\
 Given $\epsilon >0$ and for an arbitrary $\phi \in L^{\infty}(\mathbb{R}_{+})$, then we can choose $a>0$ large enough such that
\begin{align}\label{conv000}
 k_2\|\phi\|_{L^{\infty}(0,W)}[V(T)^2+ \|g\|^2_{L^1(\mathbb{R}_{+}, (1+z)dz)}](1+a)^{\beta -1}< \frac{ \epsilon}{2}.
\end{align}
For $g\in S^+,$ we define the operator $A_1$ by
\begin{align}\label{conv0}
A_1(g)(z,t):=\int_0^a \phi(z)C(z, z_1)g(z_1,t)dz_1.
\end{align}
 For $z\in (0,W)$ a.e., the function $\phi_z$ defined by
\begin{align*}
 \phi_z(\cdot):= \chi_{(0,a)}(\cdot)\phi(z)C(z,\cdot)\ \ \text{is in}\ L^{\infty}(\mathbb{R}_{+}).
\end{align*}
Since $g^n \rightharpoonup g$ in $L^1(\mathbb{R}_{+}, dz)$, it follows that
\begin{align}\label{conv1}
 A_1(g^n)(z,t) \to A_1(g)(z,t)\ \ \text{as}\ n \to \infty \ \text{for}\ z\in (0, W).
\end{align}
From $(H4)$ and  H\"{o}lder's inequality, we have
\begin{align}\label{conv2}
| A_1(g^n)(z,t)| & \leq 2k_2\|\phi \|_{L^{\infty}(0, W)}V(T)(1+W)\ \ \text{for a.e.}\ z\in (0,W).
\end{align}
Similarly, this can be shown for $A_1(g)$   which shows that $A_1(g^n)$ and $A_1(g)$ belong to $L^{\infty}(0,W).$ It follows from (\ref{conv1}) and Egoroff's theorem that
\begin{align}\label{conv3}
 A_1(g^n)(z,t) \to A_1(g)(z,t)\ \ \ \text{as}\ \ n \to \infty \ \text{almost uniformly on}\ (0, W),
\end{align}
that is for a given $\delta >0$ there exists a set $F \subseteq (0, W]$ such that the measure of $F$, $\delta >\mu (F)$   and  $A_1(g^n) \to A_1(g)$
uniformly on $(0,W]\setminus F$.\\
 Choose $\epsilon >0$. By Lemma \ref{compactness1} $(iii)$ and  $g^n \rightharpoonup g$ in $L^1(\mathbb{R}_{+}, dz)$, there is a $\delta >0$ such that for all $n$, we have
\begin{align}\label{conv4}
\int_E g^n(z,t)dz < \frac{\epsilon}{4\|\phi\|_{L^{\infty}(0,W)}k_2V(T)(1+W)}
\end{align}
whenever $\mu(E) < \delta$. By (\ref{conv3}), there is a set $F \subseteq (0,W]$ such that $\mu(F) < \delta$ and $A_1(g^n) \to A_1(g)$ uniformly on $(0,W]\setminus F.$ Thus
\begin{align}\label{conv5}
 A_1(g^n)(z,t) \to A_1(g)(z,t)\ \ \text{in}\ L^{\infty}((0,W]\setminus F)\ \text{as}\ n\to \infty.
\end{align}
Applying  H\"{o}lder's inequality, we estimate
\begin{align}\label{conv9}
\bigg|\int_0^W g^n(z,t)[A_1(g^n)(z,t)-A_1(g)(z,t)]dz_1dz \bigg|%\leq \bigg|\int_{(0,W]\setminus F} g^n(z,t)[A_1(g^n)(x,t)-A_1(g)(x,t)]dx \bigg|\nonumber\\
%&+\bigg|\int_{ C} g^n(z,t)[A_1(g^n)(x)-A_1(g)(x,t)]dx \bigg|\nonumber\\
\leq & \|A_1(g^n)-A_1(g)\|_{L^{\infty}((0, W]\setminus F)}\int_{(0,W]\setminus F} g^n(z,t)dz \nonumber\\
& +\|A_1(g^n)-A_1(g)\|_{L^{\infty}(F)} \int_F g^n(z,t)dz.
\end{align}
Set $E:=F$. By considering (\ref{conv2}) and (\ref{conv4}), we obtain
\begin{align}\label{convp2}
\|A_1(g^n)-&A_1(g)\|_{L^{\infty}(F)} \int_F g^n(z,t)dz %=  \sup_{x\in F} \{|A_1(g^n)(x,t)-A_1(g)(x,t)|\}\int_C g^n(z,t)dx\nonumber\\
 \leq \sup_{z\in F} \{|A_1(g^n)(z,t)+A_1(g)(z,t)|\}\int_F g^n(z,t)dz\nonumber\\
&\leq  4k_2 \| {\phi} \|_{L^{\infty}(0, W)}V(T)(1+W) \frac{\epsilon}{4 k_2V(T)(1+W)\|{\phi}\|_{L^{\infty}(0,W)}}\leq \epsilon.
\end{align}
From (\ref{conv5}), (\ref{convp2}) and Lemma \ref{compactness1} (i), we can simplified (\ref{conv9}) as \begin{align}\label{convp3}
\bigg| \int_0^W g^n(z,t)\{A_1(g^n)(z,t)& -A_1(g)(z,t)\}dz \bigg|\nonumber\\
& \leq \|A_1(g^n)-A_1(g)\|_{L^{\infty}((0,W]\setminus F)}V(T)+ \epsilon \to \epsilon \ \text{as}\ n \to \infty.
\end{align}
Since $ \epsilon >0$ is arbitrary, then we observe that
\begin{align}\label{convp4}
\bigg| \int_0^W g^n(z,t)\{A_1(g^n)(z,t)& -A_1(g)(z,t)\}dz \bigg| \to 0\ \text{as}\ n\to \infty.
\end{align}
For $A_1(g)\in L^{\infty}(0,W)$ and $g^n \rightharpoonup g$ in $ L^1(\mathbb{R}_{+}, dz)$ as $n \to \infty$ the definition of weak convergence implies that
\begin{align}\label{convp5}
\bigg| \int_0^W [g^n(z,t)-g(z,t)]A_1(g)(z,t)dz \bigg| \to 0\ \text{as}\ n\to \infty.
\end{align}
Using (\ref{convp4}) and (\ref{convp5}), we have
\begin{align}\label{conv7}
\bigg|\int_0^W \int_0^a &\phi(z)  C(z, z_1)[g^n(z,t)g^n(z_1,t)-g(z,t)g(z_1,t)]dz_1dz \bigg|\nonumber\\
%&=\bigg|\int_0^W [g^n(x)A_1(g^n)(x,t)-g(z,t)A_1(g)(x,t)]dx \bigg|\nonumber\\
%&=\bigg|\int_0^W \{g^n(z,t)[A_1(g^n)(x,t)-A_1(g)(x,t)]+[g^n(z,t)-g(z,t)]A_1(g)(x,t)\}dx \bigg|\nonumber\\
&\leq \bigg|\int_0^W g^n(z,t)[A_1(g^n)(z,t)-A_1(g)(z,t)]dz\bigg|  +\bigg| \int_0^W[g^n(z,t)-g(z,t)]A_1(g)(z,t)]dz \bigg|\nonumber\\
& \to 0\ \text{as}\ n \to {\infty}.
\end{align}
Then, by using (\ref{conv000}) and Lemma \ref{compactness1} $(i)$, we estimate the following term as
\begin{align}\label{conv6}
\bigg|\int_0^W \int_a^{\infty} \phi(z) C(z, z_1)&[g^n(z,t)g^n(z_1,t)-g(z,t)g(z_1,t)]dz_1dz \bigg|\nonumber\\
%&\leq  2k _2 \|\phi\|_{L^{\infty}(0,W)}\int_0^W \int_a^{\infty} (1+x)^{\beta}(1+y)^{\beta}[g^n(z,t)g^n(z_1,t)+g(z,t)g(z_1,t)]dz_1dz\nonumber\\
%&\leq  2k _2 \|\phi\|_{L^{\infty}(0,W)} \int_a^{\infty} (1+y)^{\beta}[V(T)g^n(y)+\|g\| g(y)]dy\nonumber\\
\leq & 2k _2 \|\phi\|_{L^{\infty}(0,W)} \int_a^{\infty} (1+z_1)^{\beta -1}[V(T)(1+z_1)g^n(z_1,t)\nonumber\\
&+\|g\|_{L^1(\mathbb{R}_{+}, (1+z)dz)} (1+z_1)g(z_1,t)]dz_1\nonumber\\
\leq & 2k _2 \|\phi\|_{L^{\infty}(0,W)} [V(T)^2+\|g\|^2_{L^1(\mathbb{R}_{+}, (1+z)dz)}](1+a)^{\beta -1} < \epsilon.
\end{align}
Next, applying $(H4)$ and Lemma \ref{compactness1} $(i)$, we consider the following integral as
\begin{align*}
 \int_{0}^{W}\int_{n-z}^{\infty}\phi(z)C(z, z_1)g^n(z,t)g^n(z_1,t)dz_1dz
\leq 2k_2\|\phi\|_{L^{\infty}(0,W)}\frac{V(T)^2}{(1+n-z)^{1-\beta}} .
\end{align*}
This implies that
\begin{align}\label{conv8}
\int_{0}^{W}\int_{n-z}^{\infty}\phi(z)C(z, z_1)g^n(z,t)g^n(z_1,t)dz_1 dz \to  0\ \mbox{as}\ n \to \infty.
\end{align}
Let us consider the following integral and use  $(H4)$ to simplify it as
\begin{align}\label{app 1}
\bigg| \int_{0}^{W}\phi(z) & [P_4(g^n)(z,t)-P_4(g)(z,t)] dz\bigg|\nonumber\\
  \leq &  \bigg| \int_{0}^{W}\int_{0}^{a} \phi(z)C(z, z_1) g^n(z,t)g^n(z_1,t)-g(z,t)g(z_1,t) dz_1dz \nonumber \\
  &+ \int_{0}^{W}\int_{a}^{\infty}\phi(z)C(z, z_1) g^n(z,t)g^n(z_1,t)-g(z,t)g(z_1,t) dz_1dz \nonumber \\
& - \int_{0}^{W}\int_{n-z}^{\infty}\phi(z)C(z, z_1)g^n(z,t)g^n(z_1,t)dz_1dz \bigg|.
\end{align}
From (\ref{conv7}), (\ref{conv6})  and (\ref{conv8}), (\ref{app 1}) implies that \begin{align*}
\left| \int_{0}^{W}\phi(z) \{P_4(g^n)(z,t)-P_4{g}(z,t)\} dz\right| < \epsilon.
\end{align*}
Since $\phi$ is an arbitrary function, therefore, we have
\begin{align*}
\lim_{n\to \infty}{ P_4(g^n)\rightharpoonup P_4(g). }
\end{align*}

We conclude that (\ref{luchon}) holds. Thus, this completes the proof of Lemma \ref{convergence lemma1}.
\end{proof}

Now we are in a position to prove Theorem \ref{existmain theorem1} by using above results.

%%%%%%%%%%%%%%%%%%%%%%%%%%%%%%%%%%%%%%%%%%%%%%%%%%%%%%%%%%%%%%%%%%
%%%%%%%%%%%%%%%%%%%%%%%%%%%%%%%%%%%%%%%%%%%%%%%%%%%%%%%%%%%%%%%%%%%
%%%%%%%%%%%%%%%%%%%%%%%%%%%%%%%%%%%%%%%%%%%%%%%%%%%%%%%%%%%%%%%%%%%%
\begin{proof}
 \textit{of Theorem \ref{existmain theorem1}}: Fix $W \in (0, n_k)$, $T>0$  and let us consider $(g^{n_k})_{n\in N}$ be an weakly convergent subsequence of the approximating solutions obtained from (\ref{equicontinuity f}). Hence, from  (\ref{equicontinuity f}) and for  $t\in [0,T]$, we get
\begin{align}\label{exist lim}
 g^{n_k}(z,t) \rightharpoonup  g(z,t)\ \  \text{in}\ \ {L^1((0,W), dz)}\ \ \text{as}\ \ n_k \to \infty.
\end{align}

 Let $ \phi \in \L^{\infty}(0,W)$ then from Lemma \ref{convergence lemma1}, we have for each $ s \in [0,t]$
\begin{align}\label{main convergence1}
\int_{0}^{W}\phi(z)[P^{n_k}(g^{n_k})(z,s)-P(g)(z,s)]dz \to 0 \ \ \text{as} \ \ n_k\to \infty.
\end{align}
In order to apply the dominated convergence theorem, the boundedness of the following integral is shown as
\begin{align}\label{exist domin}
\bigg|\int_{0}^{W} &\phi(z)[P^{n_k}(g^{n_k})(z,s)-P(g)(z,s)]dz\bigg|\nonumber\\ % \leq &  \|\phi\|_{L^{\infty}(0,W)}\int_{0}^{W}\bigg|P^{n_k}(g^{n_k}(s))(z)-P(g(s))(z)\bigg|dz\nonumber\\
  \leq &  \|\phi\|_{L^{\infty}(0,W)}\int_{0}^{W}\biggl[ \frac{1}{2}\int_0^z K(z-z_1, z_1)|g^{n_k}(z-z_1,s)g^{n_k}(z_1,s)-g(z-z_1,s)g(z_1,s)|dz_1\nonumber\\
 &+\int_0^{n_k-z} K(z, z_1) |g^{n_k}(z,s)g^{n_k}(z_1,s)-g(z,s)g(z_1,s)|dz_1+ \int_{n_k-z}^{\infty} K(z, z_1)g(z,s)g(z_1,s)dz_1\nonumber\\
 &+ \int_{z}^{n_k}\int_{0}^{n_k-z_1}B(z|z_1;z_2)C(z_1,z_2)\left|g^{n_k}(z_1,s)g^{n_k}(z_2,s)-g(z_1,s)g(z_2,s)\right|dz_2dz_1\nonumber \\  &+\int_{n_k}^{\infty}\int_0^{\max(0, n_k -z_1)}B(z|z_1;z_2)C(z_1,z_2)g(z_1,s)g(z_2,s)dz_2dz_1\nonumber\\
  &+\int_{z}^{\infty}\int_{\max(0, n_k- z_1)}^{\infty}B(z|z_1;z_2)C(z_1,z_2)g(z_1,s)g(z_2,s)dz_2dz_1\nonumber\\
  &+\int_{0}^{n_k-z}C(z, z_1)\left|g^{n_k}(z,s)g^{n_k}(z_1,s)-g(z,s)g(z_1,s)\right|dz_1+\int_{n_k-z}^{\infty}C(z, z_1)g(z,s)g(z_1,s)dz_1 \biggr]dz\nonumber\\
 \leq &  \|\phi\|_{L^{\infty}(0,W)} [1/2k_1(3V(T)^2+5\|g\|^2_{L^1(\mathbb{R}_{+}, (1+z)dz)})\nonumber\\
 &+2k_2 [(3N+2) \|g\|^2_{L^1(\mathbb{R}_{+}, (1+z)dz)}+(N+1) V(T)^2]< \infty.
\end{align}
Since the left-hand side of (\ref{exist domin}) is in $L^{1}((0,W), dz)$, then from (\ref{main convergence1}), (\ref{exist domin}) and the Lebesgue dominated convergence theorem, we obtain
\begin{align}\label{converge1}
\int_0^t\int_{0}^{W}\phi(z)[P^{n_k}(g^{n_k})(z,s)-P(g)(z,s)]dzds \to 0 \ \ \text{as} \ \ k\to \infty.
\end{align}
Since $\phi$ is arbitrary and (\ref{converge1}) holds for $ \phi \in L^{\infty}(0, W)$ as $k \to \infty$, hence, by applying Fubini's theorem, we get
\begin{align}\label{exist last1}
\int_{0}^{t}P^{n_k}(g^{n_k})(z, s)ds\rightharpoonup  \int_{0}^{t}P(g)(z, s)ds\  \text{ in}\ L^{1}((0,W), dz),\ \text{ as}\ k  \to \infty,
\end{align}
Then by the definition of  $P^{n_k}$  we obtain
\begin{align}\label{exist last2}
g^{n_k}(z,t)= \int_{0}^{t} P^{n_k}(g^{n_k})(z, s)ds+ g^{n}_0(z),\ \ \ \text{for}\ t \in [0,T]
\end{align}
and thus, it follows from (\ref{exist last1}), (\ref{exist lim}) and (\ref{exist last2}) that
\begin{align*}
\int_{0}^{W}\phi(z)g(z,t)dz = \int_{0}^{t} \int_{0}^{W} \phi(z)P(g)(z, s)dzds+\int_{0}^{W}\phi(z)g_0(z)dz,
\end{align*}
for any $\phi \in L^{\infty}(0,W)$.  Hence for the arbitrariness of $T$, $W$, the uniqueness of limit and for all $\phi \in L^{\infty}(0,W)$, we have $g(z,t)$ is a solution to (\ref{cfe})--(\ref{in1}). This implies that for almost any $z \in (0,W)$, we have
\begin{align*}
g(z,t)= \int_{0}^{t}P(g)(z, s)ds+ g_0(z), \ \text{for a.e.}\ z\in(0,W).
\end{align*}

This completes the proof of the existence Theorem \ref{existmain theorem1}.
\end{proof}

 In the next section, the uniqueness of weak solutions to (\ref{cfe})--(\ref{in1}) is shown under additional restrictions on collisional and breakup kernels which is based on the integrability of higher moments. This uniqueness result is motivated from the pioneer works of the DaCosta \cite{DaCosta:1995}, Escobedo et al. \cite{Escobedo:2003},  Giri \cite{Giri:2013} and Giri et al. \cite{Giri:2011}.

\section{Uniqueness}
In this section, we establish the uniqueness of weak solutions to (\ref{cfe})--(\ref{in1}) by stating the following theorems.
\begin{thm}\label{uniquemain theorem1}
Suppose that $(H1)$--$(H6)$ and $(UH1)$ hold. Let $g$ be a solution to (\ref{cfe})--(\ref{in1}) with initial data $g_0 \in S^+$. Then the solution $g \in S^+ $ is unique.
\end{thm}
In order to prove the Theorem \ref{uniquemain theorem1}, we need to show the integrability of higher moments $ M_{\sigma} $ of the concentration or number density $g$ i.e.
\begin{align*}
\int_0^t M_{\sigma}(s)ds < \infty \ \ \ \text{for all}\ \ t\in [0,T],\ \text{where}\ \ \sigma \in (1,2),
\end{align*}
which is shown in the next theorem.
\begin{thm}\label{integrability higher}
Suppose $(H1)$--$(H6)$ and $(UH1)$ hold. Let $g\in S^+$ be any solution to (\ref{cfe})--(\ref{in1}) on $[0,T],$ $ T>0$. Then for every $\epsilon >0$ and $1+\eta >0$, we have
\begin{align*}
\int_0^t M_{2+ \eta -\epsilon}(s)ds< \infty.
\end{align*}
\end{thm}
\begin{proof}
This theorem can easily be proved by using repeated applications of the following lemma, see \cite{DaCosta:1995}.
\end{proof}
\begin{lem}\label{lem2}
Suppose  $(H1)$--$(H6)$ and $(UH1)$ hold. Let $g\in S^+$ be any solution to (\ref{cfe})--(\ref{in1}) on $[0,T],$ $ T>0$  and assume
\begin{align}\label{lemint}
\int_0^t M_{\sigma}(s)ds< \infty,\ \ \text{for all}\ t \in [0,T]\ \text{and for  some}\ \sigma \geq 1 \ \text{with}\ \sigma >\beta.
\end{align}
Then with $1+\eta > 0$ for all $t\in [0,T]$ and $\lambda \in (0, 1)$, we have
\begin{align*}
\int_0^t M_{\sigma +\eta -\beta +1}(s)ds< \infty,\ \ \text{if}\ \sigma -\beta < 1,\ \ \text{if}\ \ \lambda =\sigma -\beta.\\
\end{align*}
In case $ \lambda =1- \epsilon $, where $\epsilon >0$ in arbitrarily small. Then we get
\begin{align*}
\int_0^t M_{2+\eta -\epsilon }(s)ds< \infty.
\end{align*}
\end{lem}

\begin{proof}
Let us take some $\lambda \in (0, 1)$. Now, multiplying the weight $z^{\lambda}$ with weak formulation of (\ref{cfe})--(\ref{in1}) given in Definition \ref{def1}, then integrating with respect to $z$ from $0$ to $n$ and applying Fubini's theorem, we obtain

\begin{align}\label{inte 1}
\int_0^n z^{\lambda} [g(z,t)-g_0(z)]&dz+\int_0^t \int_0^n  \int_0^{\infty}z^{\lambda}K(z, z_1)g(z,s)g(z_1,s)dz_1dzds\nonumber\\
&+\int_0^t \int_0^n  \int_0^{\infty}z^{\lambda}C(z, z_1)g(z,s)g(z_1,s)dz_1dzds\nonumber\\
&=\frac{1}{2}\int_0^t \int_0^n  \int_0^{z}z^{\lambda}K(z-z_1, z_1)g(z-z_1,s)g(z_1,s)dz_1dzds\nonumber\\
&+\int_0^t \int_0^n \int_z^{\infty}\int_0^{\infty}z^{\lambda} B(z|z_1;z_2)C(z_1,z_2)g(z_1,s)g(z_2,s)dz_2dz_1dzds.
\end{align}
The last integral on the left-hand side on (\ref{inte 1}) is estimated, by applying Fubini's theorem and $(H4)$, as
\begin{align}\label{uni2 bo1}
\int_0^t \int_0^n &  \int_0^{\infty}z^{\lambda}C(z, z_1)g(z,s)g(z_1,s)dz_1dzds\nonumber\\
&\leq 2k_2\|g\|_{L^1(\mathbb{R}_{+},(1+z) dz)} \int_0^t\biggl[ \int_0^1 z^{\lambda}(1+z)^{\beta}g(z,s)dz+\int_1^n z^{\lambda}(1+z)^{\beta}g(z,s)dz\biggr]ds.
%&\leq 2k_2\|g\|_{L^1(\mathbb{R}_{+}, dz)} \int_0^t\biggl[ \int_0^1 z^{\lambda}2^{\beta}g(z,s)dz+\int_1^n z^{\lambda+\beta}(1+\frac{1}{z})^{\beta}g(z,s)dz\biggr]ds\nonumber\\
%&\leq 2k_2\|g\|_{L^1(\mathbb{R}_{+}, dz)} \int_0^t\biggl[ \int_0^1 z^{\lambda}2^{\beta}g(z,s)dz+\int_1^n z^{\lambda+\beta}2^{\beta}g(z,s)dz\biggr]ds\nonumber\\
%&\leq 2k_2\|g\|_{L^1(\mathbb{R}_{+}, dz)}2^{\beta} \int_0^t\biggl[ \int_0^{\infty} z^{\lambda}g(z,s)dz+\int_0^{\infty} z^{\lambda+\beta}g(z,s)dz\biggr]ds\nonumber\\
\end{align}
In case $\lambda + \beta \leq \sigma $ and from assumption (\ref{lemint}), we have
\begin{align}\label{uni2 bo10}
\int_0^t \int_0^n   \int_0^{\infty} & z^{\lambda}C(z, z_1)g(z,s)g(z_1,s)dz_1dzds \nonumber\\
& \leq 2^{\beta +1}k_2\|g\|_{L^1(\mathbb{R}_{+},(1+z) dz)} \int_0^t[  M_{\lambda}(s)+ M_{\lambda+\beta}(s)]ds< \infty.
\end{align}
Similarly, by using $(H3)$, and $\lambda + \omega \leq \sigma $ and (\ref{lemint}), gives
\begin{align}\label{uni2 bo11}
\int_0^t \int_0^n \int_0^{\infty} z^{\lambda} K(z, z_1) & g(z,s)g(z_1,s)dz_1dzds\nonumber\\
 \leq & 2^{\omega +1}k_1\|g\|_{L^1(\mathbb{R}_{+},(1+z) dz)} \int_0^t[  M_{\lambda}(s)+ M_{\lambda+\omega}(s)]ds < \infty.
\end{align}

From (\ref{uni2 bo10}) and (\ref{uni2 bo11}), the finiteness of the left-hand side of (\ref{inte 1}) is cleared. Therefore, the right-hand side of (\ref{inte 1}) is bounded uniformly with respect to $n$, which guarantees
\begin{align*}
\int_0^t\int_0^{\infty}\int_0^{z_1}\int_0^{\infty}z^{\lambda} B(z|z_1;z_2)C(z_1,z_2)g(z_1,s)g(z_2,s)dz_2dzdz_1ds< \infty.
\end{align*}
Next, for any $z_1 \geq 1,$ let us consider above integral by using $(UH1)$ as
\begin{align}\label{intlem 21}
\int_0^t\int_0^{\infty}&\int_0^{z_1}\int_0^{\infty}z^{\lambda}B(z|z_1;z_2)C(z_1,z_2)g(z_1,s)g(z_2,s)dz_2dzdz_1ds \nonumber\\
 & \geq B_a \int_0^t\int_1^{\infty}\int_0^{z_1}\int_0^{\infty}z^{\lambda} {z_1}^{\eta}{z_2}^{1+\eta}g(z_1,s)g(z_2,s)dz_2dzdz_1ds \nonumber\\
% &\geq B_a\int_0^t\int_0^{\infty} \int_0^{\infty}\frac{z^{\lambda +1} }{\lambda +1}\bigg|_0^{z_1}g(z_1,s)g(z_2,s)dz_2dzdz_1ds\nonumber\\
   &\geq B_a  \int_0^t \int_1^{\infty} \int_0^{\infty} \frac{{z_1}^{\lambda +\eta +1}}{\lambda  +1} {z_2}^{1+\eta} g(z_1,s)g(z_2,s)dz_2dz_1ds\nonumber\\
 %   &=\frac{\|g\|_{L^1(\mathbb{R}_{+}, dz)}B_a  }{\lambda+1} \int_0^t\int_0^{\infty}{z_1}^{\lambda + \eta +1}g(z_1,s)dz_1ds\nonumber\\
   & \geq \frac{ B_a \inf_{t\in [0, T]}M_{1+\eta}(t) }{\lambda+1}  \int_0^t  \int_1^{\infty} {z_1}^{\lambda + \eta +1}g(z_1,s)dz_1 ds,
\end{align}
where $0<1+\eta=\frac{\alpha+\beta}{2}<1$ (from $(UH1)$ and $(H5)$). From (\ref{intlem 21}), it can easily be shown that
\begin{align}\label{intlem 2}
 \int_0^t M_{\lambda + \eta +1}(s)ds< \infty.
\end{align}

% Let $\beta +\lambda =\sigma$.\\
Then, two cases arise.\\
Case 1: For $\sigma -\beta < 1 $, if $\lambda = \sigma -\beta $, then from (\ref{intlem 2}), we get
\begin{align*}
 \int_0^t M_{\sigma -\beta + \eta +1}(s)ds< \infty.
\end{align*}
Otherwise the condition $\lambda < 1$, is more restrictive, i.e. we may take $\lambda =1-\epsilon$ for any $\epsilon >0.$ This gives
\begin{align*}
 \int_0^t M_{2+ \eta -\epsilon}(s)ds< \infty.
\end{align*}
This completes the proof of Lemma \ref{lem2}.
\end{proof}

%%%%%%%%%%%%%%%%%%%%%%%%%%%%%%%%%%%%%%%%%%%%%%%%%%%%%%%%%%%%%%%%%%%%

%%%%%%%%%%%%%%%%%%%%%%%%%%%%%%%%%%%%%%%%%%%%%%%%%%%%%%%%%%%%%%%%%%%%
\begin{proof}
\textit {of the Theorem \ref{uniquemain theorem1}:}
Let $g$ and $h$ be two weak solutions to (\ref{cfe})--(\ref{in1}) on $[0,T]$, where $T>0 $, with $ g_0=h_0$. Set $ Z:=g-h$. For $n= 1,2,3\cdots,$ we define
\begin{align}\label{uni2 1}
u^{n}(t):= \int_{0}^{n}(1+z) |Z(z,t)|dz.
\end{align}
Multiplying $|Z|$ by $(1+z)$ and using Definition \ref{def1} $(iii)$,  we get
\begin{align}\label{uni2 2}
u^{n}(t)
%= \int_{0}^{n}(1+z) \mbox{sgn}(Z(z,t))Z(z,t)dz
=\int_{0}^{n}(1+z) \mbox{sgn}(Z(z,t)) [g(z,t)-h(z,t)]dz.
\end{align}

\begin{align*}
g(z,s)g(z_1,s)-h(z,s)h(z_1,s)= g(z,s)Z(z_1,s)+h(z_1,s)Z(z,s),
\end{align*}
 we have
\begin{align}\label{uni2 4}
u^{n}(t)=& \int_{0}^{t}\int_{0}^{n}\int_{0}^{n-z} \bigg[ \frac{1}{2}(1+z+z_1) \mbox{sgn}(Z(z+z_1,s)) -(1+z)\mbox{sgn}(Z(z,s)) \bigg]\nonumber\\
 & \times K(z, z_1)[g(z,s)Z(z_1,s)+h(z_1,s)Z(z,s)]dz_1dzds\nonumber\\
& -\int_{0}^{t}\int_{0}^{n}\int_{n-z}^{\infty}(1+z) \mbox{sgn}(Z(z,s))K(z, z_1)[g(z,s)Z(z_1,s)+h(z_1,s)Z(z,s)]dz_1dzds\nonumber\\
&+ \int_{0}^{t}\int_{0}^{n}\int_{0}^{z_1}\int_{0}^{\infty}(1+z) \mbox{sgn}(Z(z,s))B(z|z_1;z_2)C(z_1,z_2)\nonumber\\
&\hspace{2cm}\times [g(z_1,s)Z(z_2, s)+h(z_2, s)Z(z_1, s)]dz_2dzdz_1ds\nonumber\\
&+ \int_{0}^{t}\int_{n}^{\infty}\int_{0}^{n}\int_{0}^{\infty}(1+z) \mbox{sgn}(Z(z,s))B(z|z_1;z_2)C(z_1,z_2)\nonumber\\
&\hspace{2cm}\times [g(z_1,s)Z(z_2, s)+h(z_2, s)Z(z_1,s))]dz_2dzdz_1ds\nonumber\\
& -\int_{0}^{t}\int_{0}^{n}\int_{0}^{\infty}(1+z) \mbox{sgn}(Z(z,s))C(z, z_1)[g(z,s)Z(z_1,s)+h(z_1,s)Z(z,s)]dz_1dzds.
\end{align}
Let us define $p$ by
\begin{align*}
p(z,z_1,t):=(1+z+z_1) \mbox{sgn}(Z(z+z_1,s)) -(1+z)\mbox{sgn}(Z(z,s))-(1+z_1)\mbox{sgn}(Z(z_1,s)).
\end{align*}

By using (\ref{N1}), (\ref{mass1}), properties of signum function and Fubini's theorem, (\ref{uni2 4}) can be written as
\begin{align}\label{uni2 5}
u^{n}(t) \leq & \int_0^t \bigg[\frac{1}{2}\int_{0}^{t}\int_{0}^{n}\int_{0}^{n-z} p(z,z_1,s) K(z, z_1)g(z,s)Z(z_1,s)dz_1dz\nonumber\\
 &+\frac{1}{2} \int_{0}^{n}\int_{0}^{n-z} p(z,z_1,s) K(z, z_1)h(z_1,s)Z(z,s)dz_1dz\nonumber\\
&+ (1+N) \int_{0}^{n}\int_{0}^{\infty}|Z(z_1,s)|C(z, z_1)g(z,s)dz_1dz\nonumber\\
&+ 2\int_{0}^{n}\int_{0}^{\infty}z|Z(z_1,s)|C(z, z_1)g(z,s)dz_1dz + N \int_{0}^{n}\int_{0}^{\infty}|Z(z,s)| C(z, z_1)h(z_1,s)dz_1dz\nonumber\\
&+ \int_{0}^{n}\int_{0}^{\infty}z|Z(z,s)|C(z, z_1)h(z_1,s)dz_1dz -\int_{0}^{n}\int_{0}^{\infty}z|Z(z,s)|C(z, z_1)h(z_1,s)dz_1dz\nonumber\\ &+\int_{n}^{\infty}\int_{0}^{n}\int_{0}^{\infty}(1+z)|Z(z_2, s)| B(z|z_1;z_2)C(z_1,z_2)g(z_1,s)dz_2dzdz_1\nonumber\\
&+ \int_{n}^{\infty}\int_{0}^{n}\int_{0}^{\infty}(1+z)|Z(z_1,s))| B(z|z_1;z_2)C(z_1,z_2)h(z_2, s)dz_2dzdz_1\nonumber\\
& +\int_{0}^{n}\int_{n-z}^{\infty}(1+z)|Z(z_1,s)| K(z, z_1)g(z,s)dz_1dz \bigg] ds:= \sum_{i=1}^{10} S_{i}^{n}(t),
\end{align}
where $S_{i}^{n}(t)$, for $i=1,2,\cdots, 10,$ are the corresponding integrals in preceding lines. %We now evaluate all  $S_{i}^{n}(t)$ as follows.
We estimate $S_1^n(t)$ and $S_2^n(t)$ as given in Giri \cite{Giri:2013}, which are
 \begin{align*}
S_1^n(t)\leq \int_0^t \Psi_{g_1}(s)u^n(s)ds,
\end{align*}
 where $\Psi_{g_1}(s):= 2^{1+\omega}k_1[M_0(g(s))+M_{1+\omega}(g(s))]$ and
  \begin{align*}
S_2^n(t)\leq \int_0^t \Psi_{g_2}(s)u^n(s)ds,
\end{align*}
 where $\Psi_{g_2}(s):= 2^{1+\omega}k_1[M_0(h(s))+M_{1+\omega}(h(s))]$.
Next, estimate $S_3^n(t)$ is evaluated, by using the integrability of higher moments and $(H4)$, as
\begin{align*}
S_3^n(t) %=&(1+N)\int_{0}^{t}\int_{0}^{n}\int_{0}^{\infty}|Z(z_1,s)|C(z, z_1)g(z,s)dz_1dzds\\
\leq &  2k_2(1+N)\int_{0}^{t}\int_{0}^{n}\int_{0}^{n}(1+z)(1+z_1)  |Z(z_1,s)|g(z,s)dz_1dzds\\
& + 2k_2(1+N)\int_{0}^{t}\int_{0}^{n}\int_{n}^{\infty} (1+z)(1+z_1)  |Z(z_1,s)|g(z,s)dz_1dzds\\
\leq &  2k_2(1+N)  \|g\|_{L^1(\mathbb{R}_{+},(1+z)dz)} \int_{0}^{t} u^n(s)  ds\nonumber\\
& +2k_2(1+N)\int_{0}^{t}\int_{0}^{n}\int_{n}^{\infty}|Z(z_1,s)|(1+z)(1+z_1)g(z,s)dz_1dzds.
\end{align*}
Further, the finiteness of the second term on the right-hand side in above inequality is shown as \begin{align*}
 2k_2(1+N)\int_{0}^{t}&\int_{0}^{n}\int_{n}^{\infty}|Z(z_1,s)|(1+z)(1+z_1)g(z,s)dz_1dzds\nonumber\\
\leq & 2k_2(1+N) \|g\|_{L^1(\mathbb{R}_{+},(1+z)dz)}( \|g\|_{L^1(\mathbb{R}_{+},(1+z)dz)}+ \|h\|_{L^1(\mathbb{R}_{+},(1+z)dz)})T.
\end{align*}
Then as $n\to \infty$, the above term goes to $0$. Hence, $S_3^n(t)$ can be rewritten as
\begin{align*}
S_3^n(t) \leq  2k_2(1+N) \|g\|_{L^1(\mathbb{R}_{+},(1+z)dz)} \lim_{n \to \infty}\int_{0}^{t}  u^n(s)ds.
\end{align*}
From the integrability of higher moments and $(UH1)$, let us now estimate $S_4^n(t)$ as
 \begin{align*}
S_4^n(t) %=&2\int_{0}^{t}\int_{0}^{n}\int_{0}^{\infty}z|Z(z_1,s)|C(z, z_1)g(z,s)dz_1dzds\\
%\leq & 2k_2\int_{0}^{t}\int_{0}^{n}\int_{0}^{n}x|Z(z_1,s)|(x^{\alpha}+x^{\beta})(1+y)g(z,s)dz_1dzds\\
 %& + 2k_2\int_{0}^{t}\int_{0}^{n}\int_{n}^{\infty}x|Z(z_1,s)|(x^{\alpha}+x^{\beta})(1+y) g(z,s)dz_1dzds\\
 \leq & 2k_2\int_{0}^{t} (M_{1+\alpha}(s)+M_{1+\beta}(s))u^n(s)ds\\
 & +2k_2\int_{0}^{t}\int_{0}^{n}\int_{n}^{\infty}z|Z(z_1,s)|(z^{\alpha}+z^{\beta})(1+z_1) g(z,s)dz_1dzds.
\end{align*}
The second integral on the right-hand side in above inequality can be further simplified as
\begin{align*}
2k_2&\int_{0}^{t}\int_{0}^{n}\int_{n}^{\infty}z|Z(z_1,s)|(z^{\alpha}+z^{\beta})(1+z_1) g(z,s)dz_1dzds\nonumber\\
&\leq 2k_2 (\|g\|_{L^1(\mathbb{R}_{+},(1+z)dz)}+\|h\|_{L^1(\mathbb{R}_{+},(1+z)dz)}) \int_{0}^{t} (M_{1+\alpha}(s)+M_{1+\beta}(s)) ds < \infty.
\end{align*}
As $n\to \infty$, the above term goes to $0$. Finally, $S_4^n(t)$ can be written as
\begin{align*}
S_4^n(t) \leq  2k_2 \lim_{n \to \infty} \int_{0}^{t} (M_{1+\alpha}(s)+M_{1+\beta}(s))u^n(s)ds.
\end{align*}
Moreover, $S_5^n(t)$ can also be estimated, by using the integrability of higher moments and $(UH1)$, as
\begin{align*}
S_5^n(t) %=&N\int_{0}^{t}\int_{0}^{n}\int_{0}^{\infty}|Z(z,s)| C(z, z_1)h(z_1,s)dz_1dzds\\
\leq & 2Nk_2\int_{0}^{t}\int_{0}^{n}\int_{0}^{\infty}|Z(z,s)|(1+z)(1+z_1)h(z_1,s)dz_1dzds\leq 2Nk_2\|h\|_{L^1(\mathbb{R}_{+},(1+z)dz)} \int_{0}^{t} u^n(s)ds.
\end{align*}
Next, we show the finiteness of $S_6^n(t)$ as \begin{align*}
S_6^n(t)%=&\int_{0}^{t}\int_{0}^{n}\int_{0}^{\infty}z|Z(z,s)|C(z, z_1)h(z_1,s)dz_1dzds\\
\leq &k_2\int_{0}^{t}\int_{0}^{n}\int_{0}^{\infty}z[g(z,s)+h(z,s)](z^{\alpha}+z^{\beta})(1+z_1)h(z_1,s)dz_1dzds\\
\leq &k_2\|g\|_{L^1(\mathbb{R}_{+},(1+z)dz)}\int_{0}^{t} (M_{1+\alpha}(g(s))+M_{1+\beta}(g(s)))(M_{1+\alpha}(h(s))+M_{1+\beta}(h(s)))ds<\infty.
\end{align*}
Since both $S_6^n(t)$ and $S_7^n(t)$ are same and finite, hence, $S_6^n(t)-S_7^n(t)$ goes to $0$, as $n\to \infty$. Now, the boundedness of $S_8^n(t)$ is shown as
\begin{align*}
S_8^n(t)%=&\int_{0}^{t}\int_{n}^{\infty}\int_{0}^{n}\int_{0}^{\infty}(1+z)|Z(z,s)| B(z|z_1;z_2)C(z_1,z_2)g(z_1,s)dz_2dzdz_1ds\nonumber\\
%\leq &\int_{0}^{t}\int_{n}^{\infty}\int_{0}^{z_1}\int_{0}^{\infty}(1+z)(g(z_2,s)+h(z,s)) B(z|z_1;z_2)C(z_1,z_2)g(z_1,s)dz_2dzdz_1ds\nonumber\\
\leq &2Nk_2\int_{0}^{t}\int_{n}^{\infty}\int_{0}^{\infty}[{z_1}^{\beta}+{z_1}^{1+\beta}](1+z)^{\beta}(g(z_2,s)+h(z,s))g(z_1,s)dz_2dz_1ds\nonumber\\
\leq &2Nk_2 (\|g\|_{L^1(\mathbb{R}_{+},(1+z)dz)}+\|h\|_{L^1(\mathbb{R}_{+},(1+z)dz)}) \int_{0}^{t}[M_{\beta}(s)+M_{1+\beta}(s) ] ds< \infty.
\end{align*}
 Similarly, the finiteness of $S_9^n(t)$ and $S_{10}^n(t)$ can be shown by using integrability of higher moments and $(UH1)$. Hence, $S_8^n(t)+ S_9^n(t) +S_{10}^n(t)\to 0$ as $n  \to \infty$.

Substituting the estimates of all $S_i^n(t)$, for $i=1,2,\cdots$, 10  into (\ref{uni2 5}), we finally obtain \begin{align*}
\int_0^{\infty}(1+z)|Z(z,t)|dz=&u(t)= \lim_{n \to \infty} u^n(t)\leq  \lim_{n\to\infty}  \sum_{i=1}^{10} S_{i}^{n}(t)
%\leq &  \lim_{n\to\infty}  \int_0^t [\Psi_{g_1}(s)+ \Psi_{g_2}(s)]u^n(s)ds\nonumber\\
%&+2(1+N)k_2 \lim_{n\to\infty} \int_{0}^{t}  [\|g\|+ (M_{1+\alpha}(s)+M_{1+\beta}(s)+ \|h\| ]u^n(s)ds\nonumber\\
\leq  \int_{0}^{t}\Psi (s)u(s)ds,
\end{align*}
where $\Psi (s):= [  \sup_{s \in [0,t]}\Psi_{g_1}(s)+ \sup_{s \in [0,t]}\Psi_{g_2}(s) +2(1+N)k_2[\|g\|_{L^1(\mathbb{R}_{+},(1+z)dz)}\\
+  \sup_{s \in [0,t]}M_{1+\alpha}(s)+ \sup_{s \in [0,t]}M_{1+\beta}(s)+2Nk_2\|h\|_{L^1(\mathbb{R}_{+},(1+z)dz)}  \geq 0$.
An application of Gronwall's inequality implies \begin{align*}
u(t)=0\ \ \ \ \text{for all}\ t \in [0, T].
\end{align*}
Therefore, $g(z,t)=h(z,t)$ for a.e., $z \in \mathbb{R}_{+}$.
This confirms the uniqueness of weak solution to (\ref{cfe})--(\ref{in1}).
\end{proof}
At last, by multiplying $z$ in (\ref{cfe}), then taking integration from $0$ to $\infty $ with respect to $z$ and finally using Fubini's theorem,
hypotheses \ref{hyp1} and the integrability of higher moments from Theorem \ref{integrability higher}, it can easily be shown that the unique solution to (\ref{cfe})--(\ref{in1}) is mass conserving, i.e.
\begin{align*}
\frac{d M_1}{dt}=0. \end{align*}

\subsection*{Acknowledgment}
The first author, PKB would like to thank University Grant Commission (UGC), $6405/11/44$, India, for assisting Ph.D fellowship and the second author, AKG wish to thank Science and Engineering Research Board (SERB), Department of Science and Technology (DST), India for their funding support through the project $YSS/2015/001306$ for completing this work.

% ------------------------------------------------------------------------
\end{document}